\documentclass[11pt]{article}

\usepackage[T1]{fontenc}
\usepackage[utf8]{inputenc}
\usepackage{lmodern}
\usepackage{microtype}
\usepackage{xspace}
\usepackage{amsmath,amssymb,amsthm,mathrsfs,mathtools}
\usepackage{enumitem}
\usepackage{booktabs,tabularx,array}
\usepackage{fullpage}
\usepackage{xcolor}
\usepackage{aliascnt}
\usepackage{hyperref}
\usepackage[nameinlink,noabbrev]{cleveref}

\allowdisplaybreaks

\definecolor{linkblue}{RGB}{28,77,121}
\hypersetup{
  colorlinks=true,
  linkcolor=linkblue,
  citecolor=linkblue,
  urlcolor=linkblue,
  pdftitle={Mobius Covariance and Coefficient Duality: From Bernoulli Series to Enumerative Applications},
  pdfauthor={Max A. Alekseyev}
}

\newtheorem{theorem}{Theorem}[section]

\newaliascnt{proposition}{theorem}
\newtheorem{proposition}[proposition]{Proposition}
\aliascntresetthe{proposition}

\newaliascnt{corollary}{theorem}
\newtheorem{corollary}[corollary]{Corollary}
\aliascntresetthe{corollary}

\newaliascnt{lemma}{theorem}

\aliascntresetthe{lemma}

\theoremstyle{definition}
\newaliascnt{definition}{theorem}
\newtheorem{definition}[definition]{Definition}
\aliascntresetthe{definition}

\newaliascnt{problem}{theorem}
\newtheorem{problem}[problem]{Problem}
\aliascntresetthe{problem}

\theoremstyle{remark}
\newaliascnt{remark}{theorem}
\newtheorem{remark}[remark]{Remark}
\aliascntresetthe{remark}

\crefname{theorem}{theorem}{theorems}
\Crefname{theorem}{Theorem}{Theorems}
\crefname{proposition}{proposition}{propositions}
\Crefname{proposition}{Proposition}{Propositions}
\crefname{corollary}{corollary}{corollaries}
\Crefname{corollary}{Corollary}{Corollaries}
\crefname{lemma}{lemma}{lemmas}
\Crefname{lemma}{Lemma}{Lemmas}
\crefname{definition}{definition}{definitions}
\Crefname{definition}{Definition}{Definitions}
\crefname{problem}{problem}{problems}
\Crefname{problem}{Problem}{Problems}
\crefname{remark}{remark}{remarks}
\Crefname{remark}{Remark}{Remarks}

\newcommand{\Bser}{\mathcal B}
\newcommand{\Eser}{\mathcal E}
\newcommand{\Pser}{\mathcal P}
\newcommand{\Mser}{\mathcal M}
\newcommand{\Cser}{\mathcal C}
\newcommand{\T}{\mathscr T}
\newcommand{\coeff}[1]{\mathop{[\![#1]\!]}\xspace}

\newcommand{\CC}{\mathbb C}
\newcommand{\litref}[2]{\cite[#2]{#1}}

\title{\textbf{M\"obius Covariance and Coefficient Duality:\\From Bernoulli Series to Enumerative Applications}}
\author{Max A. Alekseyev\footnote{The George Washington University, Washington, DC, USA. Email: \href{mailtpo:maxal@gwu.edu}{maxal@gwu.edu}}}
\date{}

\begin{document}
\maketitle

\begin{abstract}
A coefficient duality first encountered for formal Bernoulli series is shown to be equivalent to a general M\"obius covariance law for formal power series. We obtain a structural characterization, an eigenspace interpretation, and a weighted form of this duality. The Catalan convolution and Chebyshev identities from the motivating Bernoulli setting extend to arbitrary M\"obius-covariant families and yield a general Ramanujan-type summation formula encompassing consecutive half-integer powers. The framework also recovers classical Bernoulli and Euler recurrences and produces recurrence families for colored matchings and generalized central trinomial coefficients, with further realizations from reflection-symmetric Appell sequences and Gorenstein Hilbert series.
\end{abstract}

\section{Introduction}\label{sec:introduction}

The point of departure for this paper is the coefficient identity introduced in our work on Ramanujan-type formulas for sums of half-integer powers \cite{AlekseyevEtAl2025}. For a formal series $H(z)$, we write $\coeff{z^r}H(z)$ for the coefficient of $z^r$ in $H(z)$. In \cite{AlekseyevEtAl2025} we defined the \emph{formal Bernoulli series}
\[
\mathcal B_\alpha(z):=\sum_{k\ge0}\binom{\alpha}{k}B_kz^k
\qquad (\alpha\in\CC),
\]
where $B_k$ is the $k$-th Bernoulli number, and proved that, for every nonnegative integer $r$ and every $h\in\CC[[z]]$,
\begin{equation}
\coeff{z^r}h(z)\mathcal B_\alpha(z)
=
\coeff{z^r}
 h\!\left(\frac{z}{1+z}\right)
 (1+z)^{r-1-\alpha}\mathcal B_\alpha(-z).
\label{eq:intro-Bernoulli-duality}
\end{equation}
We call \eqref{eq:intro-Bernoulli-duality} a \emph{coefficient duality}: it equates two coefficient extractions related by a M\"obius substitution and its accompanying weight factor. This is Theorem~2 of \cite{AlekseyevEtAl2025}, and it is the statement generalized here.

The key observation is that the coefficient-extraction mechanism in \eqref{eq:intro-Bernoulli-duality} can be separated from its Bernoulli source. We therefore ask for which formal series $F(z)$ and weights $\alpha\in\CC$ the analogous equality
\begin{equation}
\coeff{z^r}h(z)F(z)
=
\coeff{z^r}
 h\!\left(\frac{z}{1+z}\right)
 (1+z)^{r-1-\alpha}F(-z)
\label{eq:intro-duality}
\end{equation}
holds for every nonnegative integer $r$ and every test series $h\in\CC[[z]]$. When it does, we say that $F$ satisfies the \emph{M\"obius coefficient duality of weight $\alpha$}. Table~\ref{tab:realizations} previews the principal sources of the relevant symmetry and the consequences developed from them.

\clearpage
\begin{table}[!t]
\caption{Principal realizations of the M\"obius covariance principle and their roles in the paper. Consequences marked by an asterisk are new to the best of our knowledge.}
\label{tab:realizations}
\centering
\footnotesize
\setlength{\tabcolsep}{4pt}
\renewcommand{\arraystretch}{1.18}
\begin{tabularx}{\textwidth}{>{\raggedright\arraybackslash}p{0.14\textwidth} >{\raggedright\arraybackslash}p{0.12\textwidth} >{\raggedright\arraybackslash}p{0.21\textwidth} >{\raggedright\arraybackslash}X >{\raggedright\arraybackslash\scriptsize}p{0.17\textwidth}}
\toprule
\textbf{Family} & \textbf{Series} & \textbf{Source of covariance} & \textbf{Consequences} & {\footnotesize\textbf{Representative results}}\\
\midrule
General formal series & $F(z)$ & M\"obius covariance; parity in $u=z/(2-z)$ & Parity classification; arbitrary-test-series coefficient duality\textsuperscript{*}; coordinate convolution arrays and generalized Ramanujan-type sums\textsuperscript{*}; weighted duality\textsuperscript{*} & \Cref{thm:structural-duality,thm:general-Ramanujan-sum}; \Cref{prop:coordinate-convolution,prop:weighted-duality}; \Cref{cor:Catalan-convolution}\\
Bernoulli series & $\mathcal B_\alpha(z)$ & Reflection of Bernoulli polynomials & Known Bernoulli recurrences of Chang--Ha, Kaneko, and Momiyama, and a Chang--Ha Genocchi recurrence & \Cref{cor:Chang-Ha-Bernoulli,cor:Kaneko,cor:Chang-Ha-Genocchi}; \Cref{thm:Momiyama}\\
Reflection-symmetric Appell families & $\mathcal P_{\alpha,x}(z)$ & Appell addition and reflection & Paired coefficient dualities\textsuperscript{*}; higher-order Bernoulli and Euler specializations & \Cref{thm:appell-duality}; \Cref{cor:higher-order-duality,cor:Chang-Ha-Euler}\\
Colored matchings & $\mathcal M_{\alpha,\lambda}(z)$ & Appell fixed point; Gaussian reflection & Polynomial recurrence families\textsuperscript{*}; fixed-edge refinements\textsuperscript{*} & \Cref{prop:colored-matching-duality}; \Cref{cor:colored-matching,cor:fixed-edge-refinement}\\
Weighted return walks & $\mathcal C_{\alpha;b,c}(z)$ & Laurent reflection; shifted arcsine symmetry & Central trinomial, binomial, and Delannoy recurrence families\textsuperscript{*}; fixed-step refinements\textsuperscript{*} & \Cref{prop:return-walk-duality}; \Cref{cor:trinomial-recurrence}\\
Gorenstein Hilbert series & $\varphi_A(z)$ & Stanley functional equation / palindromicity & Arbitrary-test-series coefficient dualities\textsuperscript{*}; Dehn--Sommerville specialization & \Cref{cor:Hilbert-duality,cor:Dehn-Sommerville}\\
\bottomrule
\end{tabularx}
\end{table}

Theorem~\ref{thm:structural-duality} proves that the coefficient property \eqref{eq:intro-duality} is equivalent to the covariance law $F(-z/(1-z))=(1-z)^{-\alpha}F(z)$, and hence to a parity normal form. Thus the generalization replaces the Bernoulli reflection input by abstract M\"obius covariance, while Lagrange inversion converts that covariance into the arbitrary-test-series identity \eqref{eq:intro-duality}.

In the motivating paper \cite{AlekseyevEtAl2025}, the Bernoulli duality was not an isolated auxiliary statement. It supplied the structural symmetry behind the parity vanishings in Theorem~5 and the passage between the two Catalan convolution formulas in Theorem~6; those formulas, evaluated on Chebyshev generating functions in Lemma~8, provide the key cancellation in the proof of the Ramanujan-type summation formula, Theorem~1. Once the foundational duality is generalized, the same proof architecture propagates to these dependent results. Section~\ref{sec:motivating-generalizations} gathers the resulting extensions: Corollary~\ref{cor:Catalan-convolution} generalizes the parity mechanism of Theorem~5 and the convolution duality of Theorem~6, Corollary~\ref{cor:Catalan-Chebyshev} generalizes Lemma~8, and Theorem~\ref{thm:general-Ramanujan-sum} generalizes Theorem~1. The additional vanishing at the first entry in Theorem~5 remains Bernoulli-specific. Thus the principal identities of \cite{AlekseyevEtAl2025} emerge as Bernoulli specializations of a common M\"obius-covariant structure.

While the Bernoulli series is the motivating arithmetic realization, it is the reflection formula for Bernoulli polynomials \cite[Section~24.4]{DLMF24} that supplies its M\"obius covariance. For the other families listed in Table~\ref{tab:realizations}, the covariance is supplied by different source-specific symmetries: Appell reflection, Gaussian symmetry, Laurent reflection, and Stanley-type reciprocity. In the Bernoulli case, the resulting duality also gives short proofs of recurrences of Chang--Ha, Kaneko, and Momiyama and connects, through Momiyama's argument, with Kummer's congruence.

Separating the source-specific covariance from its universal coefficient-extraction consequences also clarifies the boundary between prior work and the contribution of the present paper. Herbig, Herden, and Seaton studied the spaces $\mathcal F_s$ of formal power series satisfying
\[
\varphi\!\left(\frac{z}{z-1}\right)=(1-z)^s\varphi(z),
\]
and the graded algebra formed by these spaces \cite{HerbigHerdenSeaton2015,HerbigHerdenSeaton2021}. For integral weights they describe these covariance spaces and their invariant coordinate, and they apply the same covariance law to Gorenstein Hilbert series. Accordingly, we do not claim the M\"obius covariance law itself, or its integral-weight parity classification, as new; that part of our structural theorem is a reformulation of their theory. The additional step developed here is the arbitrary-test-series identity \eqref{eq:intro-duality}: using Lagrange inversion, we prove that M\"obius covariance is equivalent to this coefficient equality for every extraction degree $r$ and every $h\in\CC[[z]]$. The involutive eigenspace formulation, weighted duality, and recurrence-producing applications are then developed from this coefficient-extraction layer. In particular, the arbitrary-$h$ formulation and its applications below do not appear in the cited work of Herbig, Herden, and Seaton.

Unless stated otherwise, the formal theory is over $\CC$: the weight $\alpha$ may be any complex number, $r$ is a nonnegative integer, and all identities are interpreted coefficientwise in formal power series. Extra restrictions enter only when an arithmetic or combinatorial interpretation requires them. Thus the Bernoulli recurrences below specialize $\alpha$ to integral values determined by integer parameters $m,n$; coloring and step parameters are nonnegative integers when they literally count colored objects, although the resulting polynomial identities remain valid over $\CC$; and the Hilbert-series realization has an integral weight coming from the $a$-invariant and the Krull dimension. Probabilistic interpretations are stated only in the corresponding real parameter ranges.

There are several neighboring notions of duality and reciprocity. Kaneko's recurrence for Bernoulli numbers \cite[p.~192]{Kaneko1995} admits a short involutive proof due to Zagier, and Satoh transported that philosophy to $q$-Bernoulli numbers attached to formal groups \cite{Satoh2000}. Binomial-transform and Riordan-array dualities for Bernoulli and Euler sequences are studied, for example, in \cite{HeZheng2015}. Reflection-symmetric and reciprocal Appell polynomials have been developed systematically by Kellner and by Bayad--Komatsu \cite{Kellner2021Appell,Kellner2023,BayadKomatsu2017}. Matching-polynomial ``duality'' also has an established but different meaning, relating a graph and its complement \cite{Godsil1981Duality,Lass2004}. These theories provide useful context, but the coefficient-duality property \eqref{eq:intro-duality} is the organizing object here.

The paper is organized as follows. Section~\ref{sec:general-theory} develops the general theory: it proves the equivalence between M\"obius covariance, coefficient duality, and parity-normal form; derives the eigenspace, coordinate-convolution, and recurrence machinery; and establishes a weighted version of the duality. Section~\ref{sec:motivating-generalizations} applies this machinery to the Catalan coordinate from \cite{AlekseyevEtAl2025}, obtaining general versions of its convolution, parity, Chebyshev, and Ramanujan-type summation identities. Section~\ref{sec:bernoulli} treats the Bernoulli realization and its arithmetic applications. Section~\ref{sec:appell} develops reflection-symmetric Appell realizations, including higher-order Bernoulli and Euler families and an affine Euler variant. Section~\ref{sec:enumerative} gives two independent enumerative realizations, colored matchings and weighted return walks, together with their layerwise refinements and a common symmetric-moment interpretation. Section~\ref{sec:hilbert} treats Gorenstein Hilbert series and the Dehn--Sommerville specialization. Section~\ref{sec:further-directions} discusses formal groups and $q$-analogues, gives a cubic multisection as a proof of concept, and raises higher multisection questions, and Section~\ref{sec:conclusion} summarizes the resulting picture.

\section{M\"obius covariance and coefficient duality}\label{sec:general-theory}

Fix $\alpha\in\CC$ and let
\begin{equation}
\sigma(z):=-\frac{z}{1-z}.
\label{eq:sigma}
\end{equation}
Then
\begin{equation}
\sigma(\sigma(z))=z,
\qquad
1-\sigma(z)=\frac1{1-z}.
\label{eq:sigma-properties}
\end{equation}

We now separate the two properties that were intertwined in the motivating Bernoulli example.

\begin{definition}[M\"obius covariance]\label{def:covariance}
A formal power series $F(z)\in\CC[[z]]$ is \emph{M\"obius-covariant of weight $\alpha$} if
\begin{equation}
F\!\left(-\frac{z}{1-z}\right)=(1-z)^{-\alpha}F(z).
\label{eq:Fcov}
\end{equation}
\end{definition}

\begin{definition}[M\"obius coefficient duality]\label{def:coefficient-duality}
A formal power series $F(z)\in\CC[[z]]$ is said to satisfy the \emph{M\"obius coefficient duality of weight $\alpha$} if, for every nonnegative integer $r$ and every formal power series $h$,
\begin{equation}
\coeff{z^r}h(z)F(z)
=
\coeff{z^r}
 h\!\left(\frac{z}{1+z}\right)
 (1+z)^{r-1-\alpha}F(-z).
\label{eq:main-duality}
\end{equation}
\end{definition}

Definition~\ref{def:coefficient-duality} is deliberately a property, not an operation defined for arbitrary series: there is no reason for \eqref{eq:main-duality} to hold unless $F$ has the requisite symmetry. The first structural result says that this coefficient property is exactly equivalent to M\"obius covariance. It also identifies the parity normal form of every such series.

For integral weights, the equivalence between the covariance law and the parity-normal form is contained, in an equivalent formulation, in the work of Herbig, Herden, and Seaton \cite{HerbigHerdenSeaton2015,HerbigHerdenSeaton2021}. The equivalence with Definition~\ref{def:coefficient-duality} isolates the additional Lagrange-inversion structure used throughout this paper.

\begin{theorem}[Structural characterization]\label{thm:structural-duality}
Let $F(z)\in\CC[[z]]$. The following conditions are equivalent.
\begin{enumerate}[label=\textnormal{(\roman*)},leftmargin=2.4em]
\item $F$ is M\"obius-covariant of weight $\alpha$ in the sense of Definition~\ref{def:covariance}.
\item $F$ satisfies the M\"obius coefficient duality of weight $\alpha$ in the sense of Definition~\ref{def:coefficient-duality}.
\item Equation \eqref{eq:main-duality} holds for every nonnegative integer $r$ with $h=1$.
\item There exists a unique $\Phi_{F,\alpha}(t)\in\CC[[t]]$ such that
\begin{equation}
F(z)=(1-z)^{\alpha/2}
\Phi_{F,\alpha}\!\left(\left(\frac{z}{2-z}\right)^2\right).
\label{eq:Fclassification}
\end{equation}
\end{enumerate}
\end{theorem}

\begin{proof}
Assume \textnormal{(i)}. Let
\[
g(z):=\frac{z}{1-z},
\]
the compositional inverse of $z/(1+z)$. Lagrange inversion in the form
\begin{equation}
\coeff{z^r}H(z)(1+z)^r
=
\coeff{z^r}H(g(z))\frac{zg'(z)}{g(z)}
\label{eq:Lagrange-form}
\end{equation}
is standard; see, for example, \cite{Gessel2016}. Taking
\[
H(z):=h\!\left(\frac{z}{1+z}\right)(1+z)^{-1-\alpha}F(-z)
\]
and using $zg'(z)/g(z)=1/(1-z)$, the right-hand side of \eqref{eq:Lagrange-form} becomes
\[
\coeff{z^r}h(z)(1-z)^\alpha F\!\left(-\frac{z}{1-z}\right)
=\coeff{z^r}h(z)F(z),
\]
which proves \textnormal{(ii)}. The implication \textnormal{(ii)}$\Rightarrow$\textnormal{(iii)} is immediate.

Conversely, applying \eqref{eq:Lagrange-form} backwards with $h=1$ shows that the right-hand side of \eqref{eq:main-duality} equals
\[
\coeff{z^r}(1-z)^\alpha F\!\left(-\frac{z}{1-z}\right).
\]
Thus \textnormal{(iii)} for all $r$ implies equality of the two formal series
\[
F(z)=(1-z)^\alpha F\!\left(-\frac{z}{1-z}\right),
\]
which is \textnormal{(i)}.

Finally, put
\[
G(z):=(1-z)^{-\alpha/2}F(z),
\qquad
u:=\frac{z}{2-z}.
\]
Condition \textnormal{(i)} is equivalent to $G(\sigma(z))=G(z)$, while $u(\sigma(z))=-u(z)$. Hence $G$ is an even formal series in $u$, so $G(z)=\Phi_{F,\alpha}(u^2)$ for a unique $\Phi_{F,\alpha}\in\CC[[t]]$. This proves \textnormal{(i)}$\Leftrightarrow$\textnormal{(iv)}.
\end{proof}

The subscript in $\Phi_{F,\alpha}$ records that this series is attached to the fixed pair $(F,\alpha)$; for a parameterized family it generally varies with the parameter. For the Bernoulli family we abbreviate it to $\Phi_\alpha$.

\paragraph{Bernoulli normal form.}
For the formal Bernoulli series $\Bser_\alpha(z)$ from the Introduction, the parity factor in Theorem~\ref{thm:structural-duality} is explicit:
\begin{equation}
\Phi_\alpha(t):=(1-t)^{-\alpha/2}
\sum_{j\ge0}\binom{\alpha}{2j}(2-4^j)B_{2j}t^j,
\label{eq:Bernoulli-Phi}
\end{equation}
and hence
\begin{equation}
\Bser_\alpha(z)
=(1-z)^{\alpha/2}
\Phi_\alpha\!\left(\left(\frac{z}{2-z}\right)^2\right).
\label{eq:Bernoulli-normal-form}
\end{equation}
Indeed, put $u:=z/(2-z)$, so that $z=2u/(1+u)$. Then
\begin{align*}
(1-z)^{-\alpha/2}\Bser_\alpha(z)
&=(1-u^2)^{-\alpha/2}
\sum_{k\ge0}\binom{\alpha}{k}B_k(2u)^k(1+u)^{\alpha-k}.
\end{align*}
If $B_m(x):=\sum_{k=0}^m\binom{m}{k}B_kx^{m-k}$ denotes the Bernoulli polynomial, then the coefficient of $u^m$ in the last sum is
\begin{align*}
\sum_{k=0}^m
\binom{\alpha}{k}B_k2^k\binom{\alpha-k}{m-k}
&=\binom{\alpha}{m}\sum_{k=0}^m\binom{m}{k}2^kB_k\\
&=\binom{\alpha}{m}2^mB_m(1/2).
\end{align*}
Since $B_m(1/2)=(2^{1-m}-1)B_m$, all odd coefficients vanish, and the even coefficients give \eqref{eq:Bernoulli-Phi}. For example,
\[
\Phi_0(t)=1,
\qquad
\Phi_1(t)=(1-t)^{-1/2},
\qquad
\Phi_2(t)=\frac{1-t/3}{1-t}.
\]
Thus the specialization $\alpha=0$ determines only $\Phi_0$, not the series belonging to other weights.

For even integral weight there is a direct Faulhaber interpretation. Let $m\ge1$. Since
\[
\Bser_{2m}(z)=z^{2m}B_{2m}(1/z),
\]
putting $x:=1/z$ and $y:=x(x-1)$ in \eqref{eq:Bernoulli-normal-form} gives
\begin{equation}
B_{2m}(x)
=y^m\Phi_{2m}\!\left(\frac{1}{4y+1}\right).
\label{eq:Phi-Faulhaber-quadratic}
\end{equation}
Thus $\Phi_{2m}$ is a reciprocal M\"obius reparametrization of the familiar quadratic reduction of the reflection-symmetric Bernoulli polynomial $B_{2m}(x)$. In particular, if
\[
S_{2m-1}(n):=\sum_{\ell=1}^n\ell^{2m-1},
\qquad
X:=\frac{n(n+1)}2,
\]
then the Bernoulli formula for power sums and \eqref{eq:Phi-Faulhaber-quadratic} yield
\begin{equation}
S_{2m-1}(n)
=
\frac{(2X)^m\Phi_{2m}\!\left((8X+1)^{-1}\right)-B_{2m}}{2m},
\label{eq:Phi-Faulhaber}
\end{equation}
which is the classical Faulhaber polynomial in $X$. The same quadratic substitution for reflection-symmetric Appell polynomials, and its relation to Faulhaber-type polynomials, is studied systematically in \cite{Kellner2021Appell,Kellner2023}; here it appears as the reciprocal form of the M\"obius parity coordinate.

The theorem shows that the duality is not a special property of a distinguished coefficient sequence. It is the coefficient-extraction form of a M\"obius symmetry, and every series with that symmetry is obtained by choosing an arbitrary even series in the parity coordinate $u=z/(2-z)$ and restoring the weight factor $(1-z)^{\alpha/2}$.

\subsection{Eigenspaces and parity}

Fix a nonnegative integer $r$. The exponent $r-1-\alpha$ occurs repeatedly in the duality and is the shifted weight carried by the transformed test series; for brevity, write
\begin{equation}
q:=r-1-\alpha.
\label{eq:q}
\end{equation}
For a covariant series $F$, define
\begin{equation}
L_{r,F}(h):=\coeff{z^r}h(z)F(z)
\label{eq:L}
\end{equation}
and
\begin{equation}
(\T_qh)(z):=(1-z)^q h\!\left(-\frac{z}{1-z}\right).
\label{eq:T}
\end{equation}

\begin{proposition}\label{prop:eigenspace-duality}
One has
\begin{equation}
L_{r,F}(h)=(-1)^rL_{r,F}(\T_qh),
\label{eq:eigen-duality}
\end{equation}
and
\begin{equation}
\T_q^2=1.
\label{eq:T-involution}
\end{equation}
\end{proposition}

\begin{proof}
Replace $z$ by $-z$ on the right-hand side of \eqref{eq:main-duality}; then use \eqref{eq:sigma-properties} for the involution statement.
\end{proof}

\begin{corollary}\label{cor:eigen-vanish}
If $\T_qh=\varepsilon h$ with $\varepsilon\in\{1,-1\}$, then
\begin{equation}
L_{r,F}(h)=0
\qquad\text{whenever}\qquad
\varepsilon\ne(-1)^r.
\label{eq:eigen-vanish}
\end{equation}
\end{corollary}

The same coordinate that classifies $F$ diagonalizes $\T_q$. If
\begin{equation}
h(z)=(1-z)^{q/2}g\!\left(\frac{z}{2-z}\right),
\label{eq:hdiag}
\end{equation}
then
\begin{equation}
\T_qh(z)=(1-z)^{q/2}g\!\left(-\frac{z}{2-z}\right).
\label{eq:Tdiag}
\end{equation}
Thus the eigenspace decomposition is ordinary parity in $u=z/(2-z)$.

\begin{remark}
Herbig, Herden, and Seaton use the invariant
\[
\lambda_2(z):=\left(\frac{z}{z-2}\right)^2=u^2
\]
in their description of the invariant algebra \cite[Sec.~3]{HerbigHerdenSeaton2021}. Retaining the unsquared coordinate $u$ makes both eigenspaces visible and is convenient for coefficient cancellation.
\end{remark}

Taking $g(u)=u^j$ gives a universal hierarchy.

\begin{corollary}\label{cor:general-hierarchy}
Let $p:=(r-1-\alpha)/2$. If $r-j$ is odd, then
\begin{equation}
\coeff{z^r}(1-z)^p
\left(\frac{z}{2-z}\right)^jF(z)=0.
\label{eq:hierarchy}
\end{equation}
In particular, for odd $r$,
\begin{equation}
\coeff{z^r}(1-z)^{(r-1-\alpha)/2}F(z)=0.
\label{eq:master-vanishing}
\end{equation}
\end{corollary}

Writing $f_k:=\coeff{z^k}F(z)$, the case $j=0$ becomes
\begin{equation}
\sum_{k=0}^r(-1)^{r-k}
\binom{(r-1-\alpha)/2}{r-k}f_k=0,
\qquad r\ \text{odd}.
\label{eq:general-master-recurrence}
\end{equation}
More generally, since $\T_q(1-z)^s=(1-z)^{q-s}$, the function
\[
(1-z)^s-(-1)^r(1-z)^{q-s}
\]
belongs to the eigenspace opposite to $(-1)^r$. Hence, for arbitrary $s\in\CC$,
\begin{equation}
\sum_{k=0}^r(-1)^{r-k}f_k
\left\{
\binom{s}{r-k}-(-1)^r\binom{q-s}{r-k}
\right\}=0.
\label{eq:two-param-family}
\end{equation}
These identities are recurrence-producing templates: different covariant series supply different coefficient sequences $f_k$ while the kernel is fixed by the M\"obius symmetry.

\subsection{Anti-invariant coordinates and coordinate convolutions}\label{subsec:convolution-dualities}

The parity description can be packaged in a form that is useful for convolution identities. We begin by recording explicitly all series that are invariant or anti-invariant under the M\"obius involution, as well as their weighted analogues for $\T_q$.

\begin{proposition}[Anti-invariant series and weighted eigenspaces]\label{prop:anti-invariant-characterization}
Let
\[
u:=\frac{z}{2-z}.
\]
A formal power series $U(z)\in\CC[[z]]$ satisfies
\begin{equation}
U(\sigma(z))=-U(z)
\label{eq:anti-invariant-characterization}
\end{equation}
if and only if there is a unique $\Phi(t)\in\CC[[t]]$ such that
\begin{equation}
U(z)=u\,\Phi(u^2).
\label{eq:anti-invariant-normal-form}
\end{equation}
More generally, for arbitrary $q\in\CC$, the two eigenspaces of
\[
(\T_qh)(z)=(1-z)^q h(\sigma(z))
\]
are
\begin{equation}
E_q^+:=(1-z)^{q/2}\CC[[u^2]],
\qquad
E_q^-:=(1-z)^{q/2}u\,\CC[[u^2]].
\label{eq:weighted-eigenspaces-explicit}
\end{equation}
\end{proposition}

\begin{proof}
Since $u(\sigma(z))=-u(z)$ and $u=z/2+O(z^2)$, the series $u$ is a formal coordinate and hence $\CC[[z]]=\CC[[u]]$. Thus \eqref{eq:anti-invariant-characterization} says exactly that $U$, expressed as a series in $u$, is odd, which is equivalent to \eqref{eq:anti-invariant-normal-form}.

For the weighted statement, write
\[
h(z)=(1-z)^{q/2}g(u).
\]
Using $1-\sigma(z)=(1-z)^{-1}$ gives
\[
(\T_qh)(z)=(1-z)^{q/2}g(-u).
\]
Hence the $+1$ and $-1$ eigenspaces correspond respectively to the even and odd series in $u$, proving \eqref{eq:weighted-eigenspaces-explicit}.
\end{proof}

In particular, every anti-invariant series is an odd reparametrization of the canonical parity coordinate. If its linear coefficient is nonzero, then it is itself an invertible formal coordinate. Likewise, the fixed vectors of $\T_q$ are precisely the first space in \eqref{eq:weighted-eigenspaces-explicit}. This makes the following convolution construction essentially a parity pullback.

\begin{proposition}[Coordinate convolution duality]\label{prop:coordinate-convolution}
Let $r$ be a nonnegative integer, let $F$ be M\"obius-covariant of weight $\alpha$, and put $q:=r-1-\alpha$. Suppose that $U(z)\in z\CC[[z]]$ and $A(z)\in\CC[[z]]$ satisfy
\begin{equation}
U(\sigma(z))=-U(z),
\qquad
(1-z)^qA(\sigma(z))=A(z).
\label{eq:UA-symmetry}
\end{equation}
Then, for every $G(t)\in\CC[[t]]$,
\begin{equation}
\begin{aligned}
&\coeff{z^r}A(z)G(U(z))F(z)\\
&\qquad=
\coeff{z^r}A(-z)G(-U(-z))F(-z).
\end{aligned}
\label{eq:coordinate-convolution-duality}
\end{equation}
If
\begin{equation}
K_{r,j}(F;A,U):=\coeff{z^r}A(z)U(z)^jF(z),
\label{eq:coordinate-kernel}
\end{equation}
then
\begin{equation}
K_{r,j}(F;A,U)=0
\qquad\text{whenever}\qquad r+j\ \text{is odd}.
\label{eq:coordinate-kernel-vanishing}
\end{equation}
\end{proposition}

\begin{proof}
Since $\sigma(-z)=z/(1+z)$, the two relations in \eqref{eq:UA-symmetry}, evaluated at $-z$, give
\[
U\!\left(\frac{z}{1+z}\right)=-U(-z)
\]
and
\[
(1+z)^qA\!\left(\frac{z}{1+z}\right)=A(-z).
\]
Substitution into the coefficient duality \eqref{eq:main-duality} with the test series $h(z)=A(z)G(U(z))$ proves \eqref{eq:coordinate-convolution-duality}. Moreover,
\[
\T_q\bigl(A(z)U(z)^j\bigr)=(-1)^jA(z)U(z)^j,
\]
so \eqref{eq:coordinate-kernel-vanishing} follows from Corollary~\ref{cor:eigen-vanish}.
\end{proof}

Different choices in the two normal forms of Proposition~\ref{prop:anti-invariant-characterization} therefore produce different convolution kernels while retaining the same underlying M\"obius symmetry.

\subsection{A weighted form of the duality}\label{subsec:weighted-duality}

The duality admits a one-parameter deformation that couples the M\"obius symmetry with a dilation.

\begin{proposition}[Weighted duality]\label{prop:weighted-duality}
Let $r$ be a nonnegative integer, let $F$ satisfy the equivalent conditions of Theorem~\ref{thm:structural-duality}, and put
\[
p:=\frac{r-1-\alpha}{2},
\qquad c\in\CC^\times.
\]
Then
\begin{equation}
\coeff{z^r}(1+z)^pF(-z/c)
=
c^{-r}\coeff{z^r}
\bigl(1+(c-2)z-(c-1)z^2\bigr)^pF(z).
\label{eq:weighted-duality}
\end{equation}
Equivalently,
\begin{equation}
1+(c-2)z-(c-1)z^2=(1-z)(1+(c-1)z).
\label{eq:weighted-factorization}
\end{equation}
\end{proposition}

\begin{proof}
After replacing $z$ by $cw$, the left-hand side of \eqref{eq:weighted-duality} becomes
\[
c^{-r}\coeff{w^r}(1+cw)^pF(-w).
\]
Apply \eqref{eq:main-duality} with
\[
h(z)=\bigl((1-z)(1+(c-1)z)\bigr)^p.
\]
Since $2p=r-1-\alpha$,
\[
h\!\left(\frac{w}{1+w}\right)(1+w)^{2p}=(1+cw)^p,
\]
and the result follows.
\end{proof}

The value $c=1$ returns the unscaled self-dual weight, whereas $c=2$ turns the quadratic into $1-z^2$ and therefore creates a parity collapse whenever the coefficient sequence itself has a sufficiently sparse odd part. The Bernoulli application in Section~\ref{sec:bernoulli} is the basic example. More generally, the factorization \eqref{eq:weighted-factorization} suggests multisection constructions, as discussed in Section~\ref{sec:further-directions}.

\section{Generalizations of the motivating Bernoulli identities}\label{sec:motivating-generalizations}

Theorem~\ref{thm:structural-duality} generalizes the coefficient duality of Theorem~2 in \cite{AlekseyevEtAl2025}. We now keep the M\"obius-covariant family arbitrary and use the machinery of Section~\ref{sec:general-theory} to extend the other statements of that paper that are built from the same Catalan and Chebyshev structure.

\subsection{Catalan convolution and parity}\label{subsec:Catalan-specialization}

The paper \cite{AlekseyevEtAl2025} contains a particularly natural instance of the coordinate-convolution construction involving the Catalan generating series
\[
C(z):=\frac{1-\sqrt{1-4z}}{2z}.
\]
Put
\begin{equation}
U_C(z):=C(z/4)-1,
\qquad
A_C(z):=\frac{2-C(z/4)}{1-z}.
\label{eq:Catalan-UA}
\end{equation}
Since
\begin{equation}
C(z/4)=\frac{2}{1+\sqrt{1-z}},
\label{eq:Catalan-sqrt}
\end{equation}
we have
\begin{equation}
U_C(z)=\frac{1-\sqrt{1-z}}{1+\sqrt{1-z}},
\qquad
A_C(z)=\frac{C(z/4)}{\sqrt{1-z}}.
\label{eq:Catalan-UA-sqrt}
\end{equation}
The relation with the canonical parity coordinate is especially simple:
\begin{equation}
u=\frac{z}{2-z}=\frac{2U_C(z)}{1+U_C(z)^2},
\qquad
U_C(z)=\frac{u}{1+\sqrt{1-u^2}}.
\label{eq:Catalan-canonical-coordinate}
\end{equation}
Thus $U_C=u\Phi(u^2)$ with $\Phi(t):=(1+\sqrt{1-t})^{-1}$, so its anti-invariance follows immediately from Proposition~\ref{prop:anti-invariant-characterization}. Moreover,
\begin{equation}
A_C(z)=(1-z)^{-3/4}\sqrt{1-U_C(z)^2},
\label{eq:Catalan-A-normal-form}
\end{equation}
and the factor multiplying $(1-z)^{-3/4}$ is an even series in the anti-invariant coordinate $U_C$, hence an even series in $u$. By \eqref{eq:weighted-eigenspaces-explicit}, $A_C\in E_{-3/2}^+$; equivalently,
\begin{equation}
U_C(\sigma(z))=-U_C(z),
\qquad
(1-z)^{-3/2}A_C(\sigma(z))=A_C(z).
\label{eq:Catalan-eigen}
\end{equation}
Thus the Catalan pair is not an isolated functional coincidence: it is an explicit instance of the complete eigenspace classification above, with shifted weight $q=-3/2$.

\begin{corollary}[Catalan--M\"obius convolution array]\label{cor:Catalan-convolution}
Let $\{F_\alpha(z)\}$ be a family such that $F_\alpha$ is M\"obius-covariant of weight $\alpha$, and define
\begin{equation}
f_m^{(r)}:=\coeff{z^m}F_{r+1/2}(z)\qquad(m\ge0).
\label{eq:Catalan-F-row}
\end{equation}
Thus $F_{r+1/2}(z)=\sum_{m\ge0}f_m^{(r)}z^m$.
For $r\ge0$ and $0\le j\le r$, define
\begin{equation}
\kappa_{r,j}^{F}
:=
\frac{1}{r+1/2}
\coeff{z^r}
F_{r+1/2}(z)A_C(z)U_C(z)^j.
\label{eq:Catalan-kernel-general}
\end{equation}
Then, for every $G(t)\in\CC[[t]]$,
\begin{equation}
\begin{aligned}
\sum_{j=0}^{r}\kappa_{r,j}^{F}\coeff{t^j}G(t)
&=
\frac{1}{r+1/2}\coeff{z^r}
F_{r+1/2}(z)A_C(z)G(U_C(z))\\
&=
\frac{1}{r+1/2}\coeff{z^r}
F_{r+1/2}(-z)A_C(-z)G(-U_C(-z)).
\end{aligned}
\label{eq:Catalan-convolution-general}
\end{equation}
The entries of the array have the explicit form
\begin{equation}
\kappa_{r,j}^{F}
=
\frac{1}{(r+1/2)4^r}
\sum_{m=0}^{r-j}
4^m
\binom{2r-2m+1}{r-j-m}
f_m^{(r)}.
\label{eq:Catalan-kernel-explicit}
\end{equation}
Moreover,
\begin{equation}
\kappa_{r,j}^{F}=0
\qquad\text{if }r+j\text{ is odd}.
\label{eq:Catalan-kernel-zero}
\end{equation}
If $F_{r+1/2}(0)\ne0$ for every $r$, then the lower-triangular array
$\bigl(\kappa_{r,j}^{F}\bigr)_{r,j\ge0}$ is nonsingular, with diagonal entries
\begin{equation}
\kappa_{r,r}^{F}
=
\frac{F_{r+1/2}(0)}{(r+1/2)4^r}.
\label{eq:Catalan-kernel-diagonal}
\end{equation}
Thus, after separating even and odd indices, it defines two independent triangular transforms.
\end{corollary}

\begin{proof}
For $\alpha=r+1/2$ one has $q=-3/2$, so \eqref{eq:Catalan-convolution-general} and \eqref{eq:Catalan-kernel-zero} follow from Proposition~\ref{prop:coordinate-convolution}. The sum may be truncated at $j=r$ because $U_C(z)$ has order one.

For the explicit entries, use
\[
U_C(z)^j=\left(\frac z4\right)^j C(z/4)^{2j}
\]
and the Catalan coefficient identity
\begin{equation}
\frac{C(z/4)^{2j+1}}{\sqrt{1-z}}
=
\sum_{n\ge0}
4^{-n}\binom{2n+2j+1}{n}z^n,
\label{eq:Catalan-binomial-expansion}
\end{equation}
which follows directly from Lagrange inversion applied to
$U_C=(z/4)(1+U_C)^2$. Since $A_C=C(z/4)(1-z)^{-1/2}$, substitution of \eqref{eq:Catalan-F-row} and coefficient extraction give \eqref{eq:Catalan-kernel-explicit}. Setting $j=r$ yields \eqref{eq:Catalan-kernel-diagonal}, proving the final assertion.
\end{proof}

The vanishing in \eqref{eq:Catalan-kernel-zero} generalizes the parity-zero pattern in Theorem~5 of \cite{AlekseyevEtAl2025}; the additional vanishing at the first entry in that theorem is specific to the Bernoulli family. For $F_\alpha=\Bser_\alpha$, Corollary~\ref{cor:Catalan-convolution} is exactly the convolution duality appearing as Theorem~6 of \cite{AlekseyevEtAl2025}.

\subsection{Chebyshev evaluation}\label{subsec:Catalan-Chebyshev}

The same convolution array also has a simple action on the Chebyshev bases used in Lemma~8 of \cite{AlekseyevEtAl2025}. To avoid confusion with the generalized central trinomial coefficients in \Cref{sec:return-walks}, write $T_n^{\mathrm{Ch}}$ and $U_n^{\mathrm{Ch}}$ for the Chebyshev polynomials of the first and second kinds. For an indeterminate $x$, put
\begin{equation}
\begin{aligned}
\mathcal T_x(t)&:=\sum_{j\ge0}T_{2j+1}^{\mathrm{Ch}}(\sqrt{x+1})t^j
=\frac{\sqrt{x+1}(1-t)}{1-2(2x+1)t+t^2},\\
\mathcal U_x(t)&:=\sum_{j\ge0}U_{2j}^{\mathrm{Ch}}(\sqrt{x+1})t^j
=\frac{1+t}{1-2(2x+1)t+t^2}.
\end{aligned}
\label{eq:Chebyshev-generating-functions}
\end{equation}
These are simply the odd and even parts of the classical Chebyshev generating functions. The Catalan coordinate turns them into resolvents:
\begin{equation}
A_C(z)\mathcal T_x(U_C(z))=\frac{\sqrt{x+1}}{1-(x+1)z},
\qquad
A_C(-z)\mathcal U_x(-U_C(-z))=\frac{1}{1-xz}.
\label{eq:Catalan-Chebyshev-resolvents}
\end{equation}

\begin{corollary}[Chebyshev evaluation of the Catalan--M\"obius array]\label{cor:Catalan-Chebyshev}
Under the assumptions of Corollary~\ref{cor:Catalan-convolution}, let $\alpha:=r+1/2$ and define
\begin{equation}
R_{r,F}^{\pm}(x)
:=
\frac1\alpha\coeff{z^r}\frac{F_\alpha(\pm z)}{1-xz}.
\label{eq:Chebyshev-reciprocal-polynomials}
\end{equation}
Equivalently, with the notation of \eqref{eq:Catalan-F-row},
\begin{equation}
R_{r,F}^{+}(x)=\frac1\alpha\sum_{m=0}^r f_m^{(r)}x^{r-m},
\qquad
R_{r,F}^{-}(x)=\frac1\alpha\sum_{m=0}^r(-1)^m f_m^{(r)}x^{r-m}.
\label{eq:Chebyshev-reciprocal-explicit}
\end{equation}
Then
\begin{equation}
\sum_{j=0}^r\kappa_{r,j}^{F}
T_{2j+1}^{\mathrm{Ch}}(\sqrt{x+1})
=\sqrt{x+1}\,R_{r,F}^{+}(x+1),
\label{eq:Catalan-Chebyshev-T}
\end{equation}
and
\begin{equation}
\sum_{j=0}^r\kappa_{r,j}^{F}
U_{2j}^{\mathrm{Ch}}(\sqrt{x+1})
=R_{r,F}^{-}(x).
\label{eq:Catalan-Chebyshev-U}
\end{equation}
\end{corollary}

\begin{proof}
Apply \eqref{eq:Catalan-convolution-general} first with $G=\mathcal T_x$ and its first coefficient representation, and then with $G=\mathcal U_x$ and its second coefficient representation. The two identities in \eqref{eq:Catalan-Chebyshev-resolvents} give \eqref{eq:Catalan-Chebyshev-T} and \eqref{eq:Catalan-Chebyshev-U}, respectively. Formula \eqref{eq:Chebyshev-reciprocal-explicit} follows directly by expanding $(1-xz)^{-1}$.
\end{proof}

Corollary~\ref{cor:Catalan-Chebyshev} generalizes Lemma~8 of \cite{AlekseyevEtAl2025}.

\subsection{Ramanujan-type summation for covariant families}\label{subsec:general-Ramanujan}

The Chebyshev evaluation also recovers, in a general form, the summation identity that originally led to the Bernoulli duality. For integers $n\ge0$ and $m\ge3$, define the Ramanujan tail
\begin{equation}
\tau_{\mathrm R}(n,m)
:=
\sum_{\nu\ge0}
\frac{1}{\bigl(\sqrt{n+\nu}+\sqrt{n+\nu+1}\bigr)^m}.
\label{eq:Ramanujan-tail-general}
\end{equation}
This is the function denoted $\tau(n,m)$ in \cite{AlekseyevEtAl2025}; the subscript is used here to distinguish it from M\"obius involutions appearing later. The series converges for $m\ge3$.

\begin{theorem}[Ramanujan-type summation for covariant families]\label{thm:general-Ramanujan-sum}
Under the assumptions of Corollary~\ref{cor:Catalan-convolution}, let $r\ge1$, put $\alpha:=r+1/2$, and let $R_{r,F}^{\pm}$ be as in Corollary~\ref{cor:Catalan-Chebyshev}. Define
\begin{equation}
D_{r,F}(x):=R_{r,F}^{-}(x)-R_{r,F}^{+}(x),
\qquad
s_{r,F}(n):=\sqrt n\,D_{r,F}(n).
\label{eq:general-Ramanujan-summand}
\end{equation}
Then, for every positive integer $n$,
\begin{equation}
\sum_{\ell=1}^{n}s_{r,F}(\ell)
=
C_{r,F}
+\sqrt n\bigl(R_{r,F}^{-}(n)-\kappa_{r,0}^{F}\bigr)
+\sum_{j=1}^{r}\kappa_{r,j}^{F}\tau_{\mathrm R}(n,2j+1),
\label{eq:general-Ramanujan-sum}
\end{equation}
where
\begin{equation}
C_{r,F}
:=
s_{r,F}(1)-R_{r,F}^{-}(1)+\kappa_{r,0}^{F}
-\sum_{j=1}^{r}\kappa_{r,j}^{F}\tau_{\mathrm R}(1,2j+1).
\label{eq:general-Ramanujan-constant}
\end{equation}
Moreover, if $F_{r+1/2}(z)=\sum_{m\ge0}f_m^{(r)}z^m$, then
\begin{equation}
D_{r,F}(x)
=-\frac{2}{r+1/2}
\sum_{\substack{1\le m\le r\\ m\text{ odd}}}
f_m^{(r)}x^{r-m},
\label{eq:general-Ramanujan-defect}
\end{equation}
and hence the summand is the finite half-integer-power combination
\begin{equation}
s_{r,F}(n)
=-\frac{2}{r+1/2}
\sum_{\substack{1\le m\le r\\ m\text{ odd}}}
f_m^{(r)}n^{r-m+1/2}.
\label{eq:general-Ramanujan-summand-explicit}
\end{equation}
\end{theorem}

\begin{proof}
Put
\[
d_n:=\sqrt{n+1}-\sqrt n
=\frac{1}{\sqrt n+\sqrt{n+1}}.
\]
For $j\ge0$, the standard Chebyshev identity for $x^2-y^2=1$ gives
\begin{equation}
d_n^{2j+1}
=
T_{2j+1}^{\mathrm{Ch}}(\sqrt{n+1})
-\sqrt n\,U_{2j}^{\mathrm{Ch}}(\sqrt{n+1}).
\label{eq:Chebyshev-difference-power}
\end{equation}
Multiplying by $\kappa_{r,j}^{F}$, summing over $0\le j\le r$, and applying Corollary~\ref{cor:Catalan-Chebyshev} with $x=n$ yields
\begin{equation}
\sum_{j=0}^{r}\kappa_{r,j}^{F}d_n^{2j+1}
=
\sqrt{n+1}\,R_{r,F}^{+}(n+1)
-\sqrt n\,R_{r,F}^{-}(n).
\label{eq:general-Ramanujan-first-difference}
\end{equation}
Since $R_{r,F}^{+}=R_{r,F}^{-}-D_{r,F}$, this is equivalent to
\[
s_{r,F}(n+1)
=
\sqrt{n+1}\,R_{r,F}^{-}(n+1)
-\sqrt n\,R_{r,F}^{-}(n)
-\sum_{j=0}^{r}\kappa_{r,j}^{F}d_n^{2j+1}.
\]
Summing from $n=0$ to $n=N-1$ telescopes the first two terms. The $j=0$ contribution is
$\sum_{n=0}^{N-1}d_n=\sqrt N$, whereas for $j\ge1$,
\[
\sum_{n=0}^{N-1}d_n^{2j+1}
=\tau_{\mathrm R}(0,2j+1)-\tau_{\mathrm R}(N,2j+1).
\]
Hence
\[
\sum_{\ell=1}^{N}s_{r,F}(\ell)
=
-\sum_{j=1}^{r}\kappa_{r,j}^{F}\tau_{\mathrm R}(0,2j+1)
+\sqrt N\bigl(R_{r,F}^{-}(N)-\kappa_{r,0}^{F}\bigr)
+\sum_{j=1}^{r}\kappa_{r,j}^{F}\tau_{\mathrm R}(N,2j+1).
\]
At $n=0$, Equation~\eqref{eq:general-Ramanujan-first-difference} gives
$R_{r,F}^{+}(1)=\sum_{j=0}^{r}\kappa_{r,j}^{F}$. Together with
$\tau_{\mathrm R}(0,m)=1+\tau_{\mathrm R}(1,m)$ and
$s_{r,F}(1)=R_{r,F}^{-}(1)-R_{r,F}^{+}(1)$, this identifies the constant term with \eqref{eq:general-Ramanujan-constant} and proves \eqref{eq:general-Ramanujan-sum}. Finally, subtracting the two expressions in \eqref{eq:Chebyshev-reciprocal-explicit} gives \eqref{eq:general-Ramanujan-defect}, and \eqref{eq:general-Ramanujan-summand-explicit} follows from the definition of $s_{r,F}$.
\end{proof}

For the Bernoulli family, the odd-part identity
\[
\Bser_\alpha(-z)-\Bser_\alpha(z)=\alpha z
\]
gives $D_{r,\Bser}(x)=x^{r-1}$ and therefore
\[
s_{r,\Bser}(n)=n^{r-1/2}.
\]
Under $k=2r-1$, the Bernoulli coefficient definitions identify $R_{r,\Bser}^{-}$ with $P_k$ and $\kappa_{r,j}^{\Bser}$ with $A_{2j+1}^{k}$ of \cite{AlekseyevEtAl2025}. Together with the Bernoulli-specific vanishing $\kappa_{r,0}^{\Bser}=0$, Theorem~\ref{thm:general-Ramanujan-sum} therefore specializes exactly to Theorem~1 there. Thus the pure half-integer power in that theorem is a consequence of the exceptionally sparse odd part of the Bernoulli series; for a general covariant family, the natural summand is the finite combination in \eqref{eq:general-Ramanujan-summand-explicit}.

\section{Bernoulli series and arithmetic applications}\label{sec:bernoulli}

Let $B_n$ denote the Bernoulli numbers, with
\begin{equation}
\sum_{n\ge0}B_n\frac{t^n}{n!}
:=\frac{t}{e^t-1},
\qquad B_1=-\frac12,
\label{eq:Bernoulli-egf}
\end{equation}
and define
\begin{equation}
\Bser_\alpha(z):=\sum_{k\ge0}\binom{\alpha}{k}B_kz^k.
\label{eq:def-Bseries}
\end{equation}
These formal Bernoulli series were introduced in \cite{AlekseyevEtAl2025} in the study of Ramanujan's formulas for sums of positive half-integer powers. Here $\alpha\in\CC$ throughout the formal identities; the classical recurrences later in this section arise from the indicated integral specializations of $\alpha$. The functional identity underlying the duality proved there is
\begin{equation}
\Bser_\alpha\!\left(-\frac{z}{1-z}\right)
=(1-z)^{-\alpha}\Bser_\alpha(z).
\label{eq:covariance-B}
\end{equation}
For nonnegative integral $\alpha$, this follows from the reflection formula for Bernoulli polynomials \cite[Section~24.4]{DLMF24}; the general case follows coefficientwise in $\alpha$. Thus $\Bser_\alpha$ is a particular instance of Theorem~\ref{thm:structural-duality}, and its coefficient duality is
\begin{equation}
\coeff{z^r}h(z)\Bser_\alpha(z)
=
\coeff{z^r}h\!\left(\frac{z}{1+z}\right)
(1+z)^{r-1-\alpha}\Bser_\alpha(-z),
\label{eq:Bernoulli-duality}
\end{equation}
which is exactly the identity introduced in \cite{AlekseyevEtAl2025}. The published statement is formulated for real $\alpha$; for fixed $r$, polynomiality in $\alpha$ extends it coefficientwise to complex $\alpha$.

A second elementary property is
\begin{equation}
\Bser_\alpha(-z)=\Bser_\alpha(z)+\alpha z,
\label{eq:parity-B}
\end{equation}
which follows from $B_{2j+1}=0$ for $j\ge1$. Combining this sparse odd part with the general eigenspace and weighted identities gives particularly short arithmetic recurrences. For example, \eqref{eq:general-master-recurrence} specializes to
\begin{equation}
\sum_{k=0}^r(-1)^{r-k}
\binom{(r-1-\alpha)/2}{r-k}
\binom{\alpha}{k}B_k=0,
\qquad r\ \text{odd},
\label{eq:master-recurrence}
\end{equation}
and \eqref{eq:two-param-family} yields the corresponding two-parameter family with $f_k=\binom{\alpha}{k}B_k$.

\subsection{The recurrences of Chang and Ha}

\begin{corollary}[Chang--Ha Bernoulli recurrences \litref{ChangHa2001}{Theorem~1}]\label{cor:Chang-Ha-Bernoulli}
For integers $n\ge1$ and $m\ge0$,
\begin{equation}
\sum_{k=[n-m]_+}^{2n}
\binom{m+n+1}{m-n+k}
\binom{2m+k+1}{k}B_k
=0,
\label{eq:CHa}
\end{equation}
and
\begin{equation}
\sum_{k=[n-m]_+}^{2n}
\binom{m+n+1}{m-n+k}
\binom{2m+k+1}{k}
\frac{B_k}{2^k}
=
(-1)^n
\frac{m+1}{2^{2n+1}}
\binom{m+n+1}{n}.
\label{eq:CHb}
\end{equation}
Both identities follow from the general eigenspace machinery and Proposition~\ref{prop:weighted-duality}, specialized to $F=\Bser_\alpha$.
\end{corollary}

\begin{proof}[Proof of \eqref{eq:CHa}]
Put
\begin{equation}
r:=2n+1,
\qquad
\alpha:=-2m-2,
\qquad
M:=m+n+1.
\label{eq:CH-params}
\end{equation}
Then $r-1-\alpha=2M$. Taking $h(z)=(1-z)^M$ in \eqref{eq:eigen-duality}, we have
\[
\T_{2M}h=h.
\]
Since $r$ is odd,
\begin{equation}
\coeff{z^r}(1-z)^M\Bser_{-2m-2}(z)=0.
\label{eq:CH-a-coeff}
\end{equation}
Expanding gives
\[
0=
\sum_{k=0}^{2n+1}
(-1)^{2n+1-k}
\binom{M}{2n+1-k}
\binom{-2m-2}{k}B_k.
\]
Using
\begin{equation}
(-1)^k\binom{-2m-2}{k}
=
\binom{2m+k+1}{k}
\label{eq:neg-binomial}
\end{equation}
and
\[
\binom{M}{2n+1-k}
=
\binom{m+n+1}{m-n+k},
\]
we obtain \eqref{eq:CHa}. The terms below $k=[n-m]_+$ vanish because of the first binomial coefficient, while the $k=2n+1$ term vanishes because $B_{2n+1}=0$.
\end{proof}

Thus \eqref{eq:CHa} is simply the statement that an odd coefficient of a self-dual expression must vanish.

\begin{proof}[Proof of \eqref{eq:CHb}]
We use the weighted form of the duality established in \Cref{prop:weighted-duality}. With the same parameters as in \eqref{eq:CH-params}, namely
\[
r=2n+1,\qquad \alpha=-2m-2,\qquad p=M=m+n+1,
\]
and with $c=2$, it gives
\begin{equation}
\coeff{z^{2n+1}}(1+z)^M\Bser_{-2m-2}(-z/2)
=
2^{-2n-1}\coeff{z^{2n+1}}(1-z^2)^M\Bser_{-2m-2}(z).
\label{eq:CHb-weighted-step}
\end{equation}
The factor $(1-z^2)^M$ is even, while $B_k=0$ for odd $k>1$. Hence only the $B_1$-term contributes to the coefficient on the right, and since
\[
(-2m-2)B_1=m+1,
\]
we obtain
\[
\coeff{z^{2n+1}}(1-z^2)^M\Bser_{-2m-2}(z)
=(m+1)(-1)^n\binom{M}{n}.
\]
Expanding the left-hand side of \eqref{eq:CHb-weighted-step} and using \eqref{eq:neg-binomial} gives exactly \eqref{eq:CHb}.
\end{proof}

Thus the two parts of Chang and Ha's Bernoulli theorem arise from the same M\"obius mechanism. Equation \eqref{eq:CHa} is an odd coefficient in the $+1$ eigenspace, whereas \eqref{eq:CHb} is the parity-degenerate $c=2$ member of the weighted family of \Cref{prop:weighted-duality}. This also explains structurally why the factor $2^{-k}$ in \eqref{eq:CHb} produces such a simple right-hand side.

\subsection{Kaneko and a Genocchi recurrence}

\begin{corollary}[Kaneko's recurrence \litref{Kaneko1995}{p.~192}]\label{cor:Kaneko}
For every integer $n\ge1$,
\begin{equation}
\sum_{j=0}^{n}
\binom{n+1}{j}(n+j+1)B_{n+j}=0.
\label{eq:Kaneko}
\end{equation}
\end{corollary}

\begin{proof}
Set $m=0$ in \eqref{eq:CHa} and write $k=n+j$. The same specialization appears as Corollary~1(a) of Chang and Ha \cite[Corollary~1(a)]{ChangHa2001}.
\end{proof}

Chang and Ha give two Genocchi recurrences in \cite[Theorem~3]{ChangHa2001}. The one that follows directly from the Bernoulli formulas is their part~(b).

\begin{corollary}[Chang--Ha Genocchi recurrence \litref{ChangHa2001}{Theorem~3(b)}]\label{cor:Chang-Ha-Genocchi}
Let the Genocchi numbers be normalized by
\begin{equation}
G_k:=2(1-2^k)B_k.
\label{eq:Genocchi}
\end{equation}
Then, for integers $n\ge1$ and $m\ge0$,
\begin{equation}
\sum_{k=[n-m]_+}^{2n}
\binom{m+n+1}{m-n+k}
\binom{2m+k+1}{k}
\frac{G_k}{2^k}
=
(-1)^n\frac{m+1}{2^{2n}}
\binom{m+n+1}{n}.
\label{eq:Chang-Ha-Genocchi}
\end{equation}
\end{corollary}

\begin{proof}
Since
\[
\frac{G_k}{2^k}=2\left(\frac{B_k}{2^k}-B_k\right),
\]
the left-hand side of \eqref{eq:Chang-Ha-Genocchi} is twice the difference between the left-hand sides of \eqref{eq:CHb} and \eqref{eq:CHa}. The result follows immediately.
\end{proof}

\subsection{Momiyama's two-parameter recurrence}

\begin{theorem}[Momiyama's recurrence \litref{Momiyama2001}{p.~285}]\label{thm:Momiyama}
For nonnegative integers $m,n$ with $m+n>0$,
\begin{equation}
\sum_{j=0}^{m}
\binom{m+1}{j}(n+j+1)B_{n+j}
+
(-1)^{m+n}
\sum_{k=0}^{n}
\binom{n+1}{k}(m+k+1)B_{m+k}
=0.
\label{eq:Momiyama}
\end{equation}
\end{theorem}

Momiyama proved this recurrence using the Volkenborn integral and then derived in \cite[Sec.~3]{Momiyama2001} the Kummer congruence
\begin{equation}
\frac{B_n}{n}\equiv
\frac{B_{n+p-1}}{n+p-1}
\pmod{p\mathbb Z_p},
\label{eq:Kummer}
\end{equation}
for every prime $p>5$ and every integer $n$ with $1<n<p-2$. We now derive \eqref{eq:Momiyama} from the M\"obius coefficient duality.

\begin{proof}
Set
\[
N:=m+n
\]
and observe that
\begin{equation}
\Bser_{-2}(-z)
=
\sum_{t\ge0}(t+1)B_tz^t.
\label{eq:Bminus2}
\end{equation}
Define
\begin{equation}
H(z)
:=
\frac{
(1+z)^{m+1}-1
+
(-1)^N\bigl((1+z)^{n+1}-1\bigr)
}{z}.
\label{eq:H}
\end{equation}
Then the left-hand side of \eqref{eq:Momiyama} is exactly
\begin{equation}
A_{m,n}:=\coeff{z^N}H(z)\Bser_{-2}(-z).
\label{eq:A}
\end{equation}

Since $N-1-(-2)=N+1$, the general duality transforms \eqref{eq:A} into
\begin{equation}
A_{m,n}=\coeff{z^N}h(z)\Bser_{-2}(z),
\label{eq:Atrans}
\end{equation}
where
\begin{equation}
h(z)
:=
\frac{
(1-z)^{n+1}
+
(-1)^N(1-z)^{m+1}
-
(1+(-1)^N)(1-z)^{N+2}
}{z}.
\label{eq:hMom}
\end{equation}
A direct rearrangement gives
\begin{equation}
h(z)
=
(-1)^{N+1}H(-z)
+
(1+(-1)^N)
\frac{1-(1-z)^{N+2}}{z}.
\label{eq:hMom-rearrange}
\end{equation}

The first term in \eqref{eq:hMom-rearrange} contributes $-A_{m,n}$, because
\[
\coeff{z^N}H(-z)\Bser_{-2}(z)
=(-1)^NA_{m,n}.
\]
It remains to evaluate
\[
\coeff{z^N}
\frac{1-(1-z)^{N+2}}{z}
\Bser_{-2}(z).
\]
Since
\[
\Bser_{-2}(z)
=
\sum_{k\ge0}(-1)^k(k+1)B_kz^k,
\]
this coefficient is
\[
(-1)^N
\sum_{k=0}^N
(k+1)\binom{N+2}{k+1}B_k.
\]
Using
\[
(k+1)\binom{N+2}{k+1}
=(N+2)\binom{N+1}{k},
\]
we obtain
\begin{equation}
(-1)^N(N+2)
\sum_{k=0}^N
\binom{N+1}{k}B_k
=0
\label{eq:classical-B-recurrence}
\end{equation}
by the classical Bernoulli recurrence \cite[Eq.~(2)]{Momiyama2001}. Hence
\[
A_{m,n}=-A_{m,n},
\]
and therefore $A_{m,n}=0$.
\end{proof}

Momiyama notes that \eqref{eq:Momiyama} includes Kaneko's recurrence and the usual Bernoulli recurrence as specializations and then uses it to derive \eqref{eq:Kummer}; see \cite[Sec.~3]{Momiyama2001}. Thus one obtains the chain
\begin{equation}
\text{M\"obius coefficient duality}
\Longrightarrow
\text{Momiyama recurrence}
\Longrightarrow
\text{Kummer congruence},
\label{eq:chain}
\end{equation}
where the second implication is Momiyama's argument.

\section{Reflection-symmetric Appell realizations}\label{sec:appell}

The binomial-series parameter $\alpha$ remains arbitrary in $\CC$ in this section. Parameters indexing a classical polynomial family are restricted only when required by that family's conventional definition.

\subsection{Appell sequences with reflection symmetry}

Reflection-symmetric Appell sequences and their reciprocal polynomials have been studied systematically; see Kellner \cite{Kellner2021Appell} and Bayad--Komatsu \cite{BayadKomatsu2017}. In particular, the quadratic substitution associated with the reflection symmetry is a familiar feature of that theory. The point of the next result is to show how the same reflection law interacts with the binomial formal-series transform used in the present paper and hence produces a M\"obius coefficient duality.

Let $\{P_n(x)\}_{n\ge0}$ be an Appell sequence, so that
\begin{equation}
P_n(x+y)
=
\sum_{k=0}^n
\binom nkP_k(x)y^{n-k}.
\label{eq:Appell}
\end{equation}
Assume in addition that for some constant $c$,
\begin{equation}
P_n(c-x)=(-1)^nP_n(x).
\label{eq:reflection}
\end{equation}
Define
\begin{equation}
\Pser_{\alpha,x}(z)
:=
\sum_{k\ge0}
\binom{\alpha}{k}P_k(x)z^k.
\label{eq:Pseries}
\end{equation}

\begin{theorem}\label{thm:appell-duality}
Under assumptions \eqref{eq:Appell} and \eqref{eq:reflection},
\begin{equation}
\Pser_{\alpha,x}
\!\left(-\frac{z}{1-z}\right)
=
(1-z)^{-\alpha}
\Pser_{\alpha,c-1-x}(z).
\label{eq:Appell-cov}
\end{equation}
Consequently,
\begin{equation}
\coeff{z^r}h(z)\Pser_{\alpha,x}(z)
=
\coeff{z^r}
h\!\left(\frac{z}{1+z}\right)
(1+z)^{r-1-\alpha}
\Pser_{\alpha,c-1-x}(-z).
\label{eq:Appell-duality}
\end{equation}
\end{theorem}

\begin{proof}
It is enough first to prove \eqref{eq:Appell-cov} for $\alpha=N$ a nonnegative integer. By \eqref{eq:Appell},
\begin{equation}
\Pser_{N,x}(z)
=
z^NP_N\!\left(x+\frac1z\right).
\label{eq:reciprocal-Appell}
\end{equation}
Therefore
\[
\begin{aligned}
\Pser_{N,x}(\sigma(z))
&=
\left(-\frac{z}{1-z}\right)^N
P_N\!\left(x+1-\frac1z\right)\\
&=
\left(-\frac{z}{1-z}\right)^N
(-1)^N
P_N\!\left(c-1-x+\frac1z\right)\\
&=
(1-z)^{-N}
\Pser_{N,c-1-x}(z).
\end{aligned}
\]
Coefficientwise polynomiality in $\alpha$ gives the general case. Formula \eqref{eq:Appell-duality} follows from the same Lagrange-inversion argument as the covariant-series duality.
\end{proof}

For the ordinary Bernoulli and Euler polynomials one has $c=1$:
\begin{equation}
B_n(1-x)=(-1)^nB_n(x),
\qquad
E_n(1-x)=(-1)^nE_n(x).
\label{eq:BEreflection}
\end{equation}
Their Appell addition and reflection formulas are classical; see, for example, \cite{DLMF24}.

For Bernoulli polynomials, the theorem gives
\begin{equation}
\coeff{z^r}h(z)\Bser_{\alpha,x}(z)
=
\coeff{z^r}
h\!\left(\frac{z}{1+z}\right)
(1+z)^{r-1-\alpha}
\Bser_{\alpha,-x}(-z),
\label{eq:BernPoly-duality}
\end{equation}
where
\begin{equation}
\Bser_{\alpha,x}(z)
:=
\sum_{k\ge0}\binom{\alpha}{k}B_k(x)z^k.
\label{eq:BernPoly-series}
\end{equation}
At $x=0$, this recovers the Bernoulli-series realization of the general duality.

\subsection{Higher-order Bernoulli and Euler polynomials}

Let $\ell$ be a positive integer. The higher-order Bernoulli and Euler polynomials are defined by
\begin{equation}
\sum_{n\ge0}
B_n^{(\ell)}(x)\frac{t^n}{n!}
:=
\left(\frac{t}{e^t-1}\right)^\ell e^{xt},
\label{eq:higher-B}
\end{equation}
and
\begin{equation}
\sum_{n\ge0}
E_n^{(\ell)}(x)\frac{t^n}{n!}
:=
\left(\frac{2}{e^t+1}\right)^\ell e^{xt}.
\label{eq:higher-E}
\end{equation}
These standard higher-order analogues are discussed, for example, in \cite{DLMF24}. Both are Appell sequences, and their generating functions immediately imply
\begin{equation}
P_n^{(\ell)}(\ell-x)=(-1)^nP_n^{(\ell)}(x),
\label{eq:higher-reflection}
\end{equation}
where $P$ denotes either $B$ or $E$. Hence the preceding Appell theorem applies with $c=\ell$.

\begin{corollary}\label{cor:higher-order-duality}
Let $P_n^{(\ell)}(x)$ denote either the higher-order Bernoulli or higher-order Euler polynomials and define
\[
\Pser_{\alpha,x}^{(\ell)}(z)
:=
\sum_{k\ge0}
\binom{\alpha}{k}P_k^{(\ell)}(x)z^k.
\]
Then
\begin{equation}
\coeff{z^r}h(z)\Pser_{\alpha,x}^{(\ell)}(z)
=
\coeff{z^r}
h\!\left(\frac{z}{1+z}\right)
(1+z)^{r-1-\alpha}
\Pser_{\alpha,\ell-1-x}^{(\ell)}(-z).
\label{eq:higher-duality}
\end{equation}
\end{corollary}

In particular, the point
\begin{equation}
x=\frac{\ell-1}{2}
\label{eq:fixedpoint}
\end{equation}
is fixed by $x\mapsto\ell-1-x$. Thus
\begin{equation}
\sum_{k\ge0}
\binom{\alpha}{k}
P_k^{(\ell)}\!\left(\frac{\ell-1}{2}\right)z^k
\label{eq:higher-fixed-series}
\end{equation}
satisfies exactly the same self-duality as the ordinary Bernoulli series. Consequently every eigenspace recurrence of \Cref{eq:hierarchy,eq:two-param-family} immediately produces corresponding recurrences for the values
\[
B_k^{(\ell)}\!\left(\frac{\ell-1}{2}\right),
\qquad
E_k^{(\ell)}\!\left(\frac{\ell-1}{2}\right).
\]

Reciprocal relations and shortened recurrences for generalized and higher-order Bernoulli and Euler families are already well developed; see Agoh \cite{Agoh2017} and Wang--Chu \cite{WangChu2018}. Accordingly, the significance of \eqref{eq:higher-duality} is not the first occurrence of higher-order reciprocal identities, but the fact that the general M\"obius coefficient duality \eqref{eq:main-duality} persists uniformly for these Appell families.

\subsection{An affine Euler duality}

The Euler numbers in the convention used by Chang and Ha are
\begin{equation}
E_k:=2^kE_k(1/2).
\label{eq:EulerConvention}
\end{equation}
Thus it is natural to consider
\begin{equation}
\Eser_\alpha(z)
:=
\sum_{k\ge0}
\binom{\alpha}{k}
E_k(1/2)z^k.
\label{eq:Eseries}
\end{equation}
Because
\[
E_k(1-x)=(-1)^kE_k(x),
\]
one has
\begin{equation}
E_{2j+1}(1/2)=0,
\label{eq:Euler-even}
\end{equation}
so $\Eser_\alpha(z)$ is even.

The Euler difference equation
\begin{equation}
E_k(x+1)+E_k(x)=2x^k
\label{eq:Euler-difference}
\end{equation}
gives, at $x=-1/2$,
\begin{equation}
E_k(-1/2)+E_k(1/2)
=2(-1/2)^k.
\label{eq:Euler-minus-half}
\end{equation}
These identities are classical; see \cite{DLMF24}.

Applying the Appell duality with $x=1/2$, followed by \eqref{eq:Euler-minus-half}, gives the affine covariance
\begin{equation}
\Eser_\alpha
\!\left(-\frac{z}{1-z}\right)
=
(1-z)^{-\alpha}
\left\{
2(1-z/2)^\alpha-\Eser_\alpha(z)
\right\}.
\label{eq:Euler-affine-cov}
\end{equation}
The same Lagrange-inversion argument therefore gives the following.

\begin{proposition}[Affine Euler duality]\label{prop:Euler-affine-duality}
For $q=r-1-\alpha$,
\begin{equation}
\begin{aligned}
&\coeff{z^r}h(z)\Eser_\alpha(z)\\
&\quad+
\coeff{z^r}
h\!\left(\frac{z}{1+z}\right)
(1+z)^q
\Eser_\alpha(-z)
=
2\coeff{z^r}h(z)(1-z/2)^\alpha.
\end{aligned}
\label{eq:Euler-affine-duality}
\end{equation}
\end{proposition}

This affine duality immediately yields the Euler-number theorem of Chang and Ha.

\begin{corollary}[Chang--Ha Euler recurrence \litref{ChangHa2001}{Theorem~2}]\label{cor:Chang-Ha-Euler}
For $n\ge1$ and $m\ge0$,
\begin{equation}
\sum_{k=[n-m]_+}^{2n}
\binom{m+n}{m-n+k}
\binom{2m+k}{k}
\frac{E_k}{2^k}
=
\left(-\frac14\right)^n
\binom{m+n}{n}.
\label{eq:CH-Euler}
\end{equation}
\end{corollary}

\begin{proof}
Set
\[
r:=2n,
\qquad
\alpha:=-2m-1,
\qquad
M:=m+n,
\]
so that $q=2M$. Take
\[
h(z)=(1-z)^M.
\]
Then the two coefficients on the left of \eqref{eq:Euler-affine-duality} are equal: $\Eser_\alpha$ is even, $r$ is even, and the dual transform replaces $(1-z)^M$ by $(1+z)^M$. Consequently
\begin{equation}
\coeff{z^{2n}}(1+z)^M\Eser_\alpha(-z)
=
\coeff{z^{2n}}
(1-z)^M(1-z/2)^{-2m-1}.
\label{eq:Euler-step}
\end{equation}
The left-hand side of \eqref{eq:Euler-step} is exactly the sum in \eqref{eq:CH-Euler}, since
\[
(-1)^k\binom{-2m-1}{k}
=
\binom{2m+k}{k}.
\]
It remains to evaluate the right-hand side. We claim that
\begin{equation}
\coeff{z^{2n}}
(1-z)^{m+n}(1-z/2)^{-2m-1}
=
\left(-\frac14\right)^n
\binom{m+n}{n}.
\label{eq:Euler-coeff-eval}
\end{equation}
To see this, make the coefficient substitution
\[
z=\frac{2u}{1+u}.
\]
Lagrange inversion, or equivalently a residue change of variables, gives
\begin{equation}
\coeff{z^r}f(z)
=
2^{-r}
\coeff{u^r}(1+u)^{r-1}
f\!\left(\frac{2u}{1+u}\right).
\label{eq:coeff-substitution}
\end{equation}
For $r=2n$, the right side of \eqref{eq:Euler-coeff-eval} becomes
\[
4^{-n}
\coeff{u^{2n}}
(1-u)^{m+n}(1+u)^{m+n}
=
4^{-n}\coeff{u^{2n}}(1-u^2)^{m+n},
\]
which is
\[
\left(-\frac14\right)^n\binom{m+n}{n}.
\]
\end{proof}

Thus the Bernoulli and Euler recurrences of Chang and Ha are two manifestations of the same M\"obius mechanism: the Bernoulli series is exactly covariant, while the centered Euler series is affinely covariant.

\section{Enumerative realizations}\label{sec:enumerative}

\subsection{Colored matchings and a rescaled M\"obius duality}\label{sec:colored-matchings}

For $0\le j\le \lfloor n/2\rfloor$, let
\begin{equation}
m_{n,j}:=\frac{n!}{(n-2j)!j!2^j}
\label{eq:matching-number}
\end{equation}
denote the number of $j$-edge matchings of the complete graph $K_n$, and define the matching enumerator
\begin{equation}
M_n(\lambda)
:=
\sum_{j=0}^{\lfloor n/2\rfloor}m_{n,j}\lambda^j.
\label{eq:Mnlambda}
\end{equation}
For a nonnegative integer $\lambda$, $M_n(\lambda)$ counts matchings of $K_n$ in which every selected edge receives one of $\lambda$ colors, independently of the other edges. Equivalently, it counts involutions of $[n]$ with $\lambda$ colors available for each transposition. In particular, $M_n(1)$ is the ordinary involution number. The exponential formula gives
\begin{equation}
\sum_{n\ge0}M_n(\lambda)\frac{t^n}{n!}
=
\exp\!\left(t+\frac{\lambda t^2}{2}\right).
\label{eq:M-egf}
\end{equation}
The connection between complete-graph matchings and Hermite polynomials is classical; see, for example, Godsil \cite{Godsil1981Duality}.

The occurrence of four colors in the specialization considered initially is forced by the fixed point of the Appell parameter involution. Indeed, introduce the Appell sequence $\{P_n^{(\lambda)}(x)\}$ by
\begin{equation}
\sum_{n\ge0}P_n^{(\lambda)}(x)\frac{t^n}{n!}
:=
\exp\!\left(xt+\frac{\lambda t^2}{8}\right).
\label{eq:matching-Appell-egf}
\end{equation}
The factor independent of $x$ is even, so
\begin{equation}
P_n^{(\lambda)}(-x)=(-1)^nP_n^{(\lambda)}(x).
\label{eq:matching-Appell-reflection}
\end{equation}
The reflection parameter in \Cref{thm:appell-duality} is therefore $c=0$, and its induced transformation $x\mapsto-1-x$ has the unique fixed point $x=-1/2$. A direct comparison of \eqref{eq:M-egf} and \eqref{eq:matching-Appell-egf} gives
\begin{equation}
M_n(\lambda)=(-2)^nP_n^{(\lambda)}(-1/2).
\label{eq:M-Appell-relation}
\end{equation}
For $\lambda=4$, $P_n^{(4)}(x)$ is the usual involution polynomial
\[
K_n(x):=\sum_{\pi\in\operatorname{Inv}(n)}x^{\operatorname{fix}(\pi)},
\]
so that
\begin{equation}
M_n(4)=(-2)^nK_n(-1/2)
=
\sum_{\pi\in\operatorname{Inv}(n)}4^{c_2(\pi)},
\label{eq:four-colored-specialization}
\end{equation}
where $c_2(\pi)$ is the number of transpositions of $\pi$. Thus the previously observed sequence
\[
1,1,5,13,73,281,1741,\ldots
\]
is simply the specialization $M_n(4)$ of the polynomial family \eqref{eq:Mnlambda}.

Define the binomial matching series
\begin{equation}
\Mser_{\alpha,\lambda}(z)
:=
\sum_{k\ge0}\binom{\alpha}{k}M_k(\lambda)z^k.
\label{eq:Mseries}
\end{equation}
The Appell fixed point produces a rescaled M\"obius covariance and hence a second coefficient duality.

\begin{proposition}[Colored-matching duality]\label{prop:colored-matching-duality}
For all $\alpha,\lambda\in\CC$,
\begin{equation}
\Mser_{\alpha,\lambda}\!\left(-\frac{z}{1+2z}\right)
=
(1+2z)^{-\alpha}\Mser_{\alpha,\lambda}(z).
\label{eq:M-covariance}
\end{equation}
Consequently, for every formal power series $h(z)$ and every $r\ge0$,
\begin{equation}
\coeff{z^r}h(z)\Mser_{\alpha,\lambda}(z)
=
\coeff{z^r}
h\!\left(\frac{z}{1-2z}\right)
(1-2z)^{r-1-\alpha}
\Mser_{\alpha,\lambda}(-z).
\label{eq:M-duality}
\end{equation}
\end{proposition}

\begin{proof}
Let
\[
\mathcal P^{(\lambda)}_{\alpha,x}(z)
:=
\sum_{k\ge0}\binom{\alpha}{k}P_k^{(\lambda)}(x)z^k.
\]
By \eqref{eq:M-Appell-relation},
\begin{equation}
\Mser_{\alpha,\lambda}(z)
=
\mathcal P^{(\lambda)}_{\alpha,-1/2}(-2z).
\label{eq:M-P-scaling}
\end{equation}
At the fixed point $x=-1/2$, \eqref{eq:Appell-cov} becomes
\[
\mathcal P^{(\lambda)}_{\alpha,-1/2}\!\left(-\frac{w}{1-w}\right)
=
(1-w)^{-\alpha}\mathcal P^{(\lambda)}_{\alpha,-1/2}(w).
\]
Putting $w=-2z$ and using \eqref{eq:M-P-scaling} gives \eqref{eq:M-covariance}. Rescaling \eqref{eq:Appell-duality} in the same way gives \eqref{eq:M-duality}.
\end{proof}

The series \eqref{eq:Mseries} also has the explicit expansion
\begin{align}
\Mser_{\alpha,\lambda}(z)
&=
(1+z)^\alpha
\sum_{j\ge0}
\frac{\alpha^{\underline{2j}}}{j!}
\left(\frac{\lambda z^2}{2(1+z)^2}\right)^j
\label{eq:M-series-explicit}\\
&=
(1+z)^\alpha\,
{}_2F_0\!\left(
-\frac{\alpha}{2},\frac{1-\alpha}{2};-;
2\lambda\left(\frac{z}{1+z}\right)^2
\right),
\nonumber
\end{align}
where $\alpha^{\underline{j}}:=\alpha(\alpha-1)\cdots(\alpha-j+1)$. In this form the covariance is transparent: the rational coordinate
\begin{equation}
u:=\frac{z}{1+z}
\label{eq:M-parity-coordinate}
\end{equation}
changes sign under
\begin{equation}
\tau(z):=-\frac{z}{1+2z},
\qquad
\frac{\tau(z)}{1+\tau(z)}=-\frac{z}{1+z}.
\label{eq:M-tau}
\end{equation}
Thus the matching duality has the same parity mechanism as the Bernoulli duality, but in a different M\"obius coordinate.

More explicitly, put $\rho:=r-1-\alpha$ and define
\begin{equation}
(\mathscr S_\rho h)(z)
:=
(1+2z)^\rho h\!\left(-\frac{z}{1+2z}\right).
\label{eq:S-operator}
\end{equation}
Then $\mathscr S_\rho^2=1$, and \eqref{eq:M-duality} is equivalent to
\begin{equation}
\coeff{z^r}h(z)\Mser_{\alpha,\lambda}(z)
=
(-1)^r
\coeff{z^r}(\mathscr S_\rho h)(z)\Mser_{\alpha,\lambda}(z).
\label{eq:M-eigenspace-duality}
\end{equation}
If
\begin{equation}
h(z)=(1+2z)^{\rho/2}
g\!\left(\frac{z}{1+z}\right),
\label{eq:M-diagonal-form}
\end{equation}
then
\[
(\mathscr S_\rho h)(z)
=
(1+2z)^{\rho/2}
g\!\left(-\frac{z}{1+z}\right).
\]
Hence $\mathscr S_\rho$ is again conjugate to ordinary parity.

\begin{corollary}[Colored-matching recurrence]\label{cor:colored-matching}
Let $r$ be odd and $\alpha,\lambda\in\CC$. Then
\begin{equation}
\sum_{k=0}^r
2^{r-k}
\binom{(r-1-\alpha)/2}{r-k}
\binom{\alpha}{k}
M_k(\lambda)
=0.
\label{eq:colored-matching-general}
\end{equation}
In particular, for $m,n\ge0$,
\begin{equation}
\sum_{k=[n-m]_+}^{2n+1}
(-1)^k2^{2n+1-k}
\binom{m+n+1}{m-n+k}
\binom{2m+k+1}{k}
M_k(\lambda)
=0.
\label{eq:colored-matching-mn}
\end{equation}
\end{corollary}

\begin{proof}
For
\[
p:=\frac{r-1-\alpha}{2},
\]
the function $(1+2z)^p$ belongs to the $+1$ eigenspace of $\mathscr S_{r-1-\alpha}$. Since $r$ is odd, \eqref{eq:M-eigenspace-duality} gives
\[
\coeff{z^r}(1+2z)^p\Mser_{\alpha,\lambda}(z)=0.
\]
Expansion gives \eqref{eq:colored-matching-general}. Taking $r=2n+1$ and $\alpha=-2m-2$ gives $p=m+n+1$ and
\[
\binom{-2m-2}{k}=(-1)^k\binom{2m+k+1}{k},
\]
which yields \eqref{eq:colored-matching-mn}.
\end{proof}

The identity is polynomial in $\lambda$, and in fact its cancellation occurs separately for every possible number of edges.

\begin{corollary}[Fixed-edge refinement]\label{cor:fixed-edge-refinement}
Let $r$ be odd, let $0\le j\le\lfloor r/2\rfloor$, and put $p:=(r-1-\alpha)/2$. Then
\begin{equation}
\sum_{k=2j}^r
2^{r-k}
\binom{p}{r-k}
\binom{\alpha}{k}
m_{k,j}
=0,
\label{eq:fixed-edge-recurrence}
\end{equation}
where $m_{k,j}$ is given by \eqref{eq:matching-number}.
\end{corollary}

\begin{proof}
Extracting the coefficient of $\lambda^j$ in \eqref{eq:colored-matching-general} immediately gives \eqref{eq:fixed-edge-recurrence}. There is also a direct parity proof. From \eqref{eq:M-series-explicit},
\begin{equation}
\coeff{\lambda^j}\Mser_{\alpha,\lambda}(z)
=
\frac{\alpha^{\underline{2j}}}{j!2^j}
 z^{2j}(1+z)^{\alpha-2j}.
\label{eq:lambda-layer}
\end{equation}
Set $R:=r-2j$. Since $r$ is odd, $R$ is odd, and the relevant coefficient is a constant multiple of
\[
\coeff{z^R}(1+2z)^p(1+z)^{R-1-2p}.
\]
Under the coefficient change of variable $u=z/(1+z)$, or equivalently by residues, this becomes
\begin{equation}
\coeff{u^R}(1-u^2)^p=0,
\label{eq:fixed-edge-parity}
\end{equation}
because $R$ is odd.
\end{proof}

As in \Cref{sec:general-theory}, one may use both eigenspaces rather than only the fixed vector $(1+2z)^{\rho/2}$. Since
\[
\mathscr S_\rho(1+2z)^s=(1+2z)^{\rho-s},
\]
one obtains, for arbitrary $s\in\CC$,
\begin{equation}
\sum_{k=0}^r
2^{r-k}\binom{\alpha}{k}M_k(\lambda)
\left\{
\binom{s}{r-k}
-(-1)^r\binom{\rho-s}{r-k}
\right\}
=0,
\qquad \rho=r-1-\alpha.
\label{eq:colored-matching-two-eigenspace}
\end{equation}
Thus the matching example reproduces the full recurrence-producing eigenspace mechanism, not merely one isolated identity.

\subsubsection*{A symmetric-moment interpretation}

For real $\lambda\ge0$, let $X$ be a Gaussian random variable with mean $1$ and variance $\lambda$. Its moment generating function is
\[
\mathbb E(e^{tX})=\exp\!\left(t+\frac{\lambda t^2}{2}\right),
\]
so \eqref{eq:M-egf} gives
\begin{equation}
M_n(\lambda)=\mathbb E(X^n),
\qquad
\Mser_{\alpha,\lambda}(z)=\mathbb E(1+zX)^\alpha,
\label{eq:Gaussian-moments}
\end{equation}
with the latter interpreted coefficientwise as a formal power series. Since $X$ is symmetric about $1$, one has $X\overset{d}=2-X$, and therefore
\[
\begin{aligned}
\Mser_{\alpha,\lambda}\!\left(-\frac{z}{1+2z}\right)
&=
(1+2z)^{-\alpha}
\mathbb E\bigl(1+z(2-X)\bigr)^\alpha\\
&=
(1+2z)^{-\alpha}\Mser_{\alpha,\lambda}(z),
\end{aligned}
\]
which gives a second, non-Appell explanation of \eqref{eq:M-covariance}.

The Gaussian distribution is not essential. The covariance is a formal consequence of reflection symmetry of moments.

\begin{proposition}[Symmetric-moment duality]\label{prop:symmetric-moments}
Let $\beta\in\CC$, and let $\Lambda:\CC[x]\to\CC$ be a linear functional satisfying
\begin{equation}
\Lambda(f(x))=\Lambda(f(2\beta-x))
\qquad(f\in\CC[x]).
\label{eq:symmetric-functional}
\end{equation}
Put $\mu_n:=\Lambda(x^n)$ and define
\begin{equation}
\mathcal U_{\alpha,\beta}(z)
:=
\sum_{n\ge0}\binom{\alpha}{n}\mu_nz^n.
\label{eq:symmetric-moment-series}
\end{equation}
Then
\begin{equation}
\mathcal U_{\alpha,\beta}\!\left(-\frac{z}{1+2\beta z}\right)
=
(1+2\beta z)^{-\alpha}\mathcal U_{\alpha,\beta}(z),
\label{eq:symmetric-moment-covariance}
\end{equation}
and consequently
\begin{equation}
\coeff{z^r}h(z)\mathcal U_{\alpha,\beta}(z)
=
\coeff{z^r}
h\!\left(\frac{z}{1-2\beta z}\right)
(1-2\beta z)^{r-1-\alpha}
\mathcal U_{\alpha,\beta}(-z).
\label{eq:symmetric-moment-duality}
\end{equation}
\end{proposition}

\begin{proof}
Extend $\Lambda$ coefficientwise to formal power series in $z$. Then
\[
\begin{aligned}
\mathcal U_{\alpha,\beta}\!\left(-\frac{z}{1+2\beta z}\right)
&=
\Lambda\!\left(1-\frac{zx}{1+2\beta z}\right)^\alpha\\
&=
(1+2\beta z)^{-\alpha}
\Lambda\bigl(1+z(2\beta-x)\bigr)^\alpha\\
&=
(1+2\beta z)^{-\alpha}\Lambda(1+zx)^\alpha,
\end{aligned}
\]
which proves \eqref{eq:symmetric-moment-covariance}. The coefficient duality follows from the same Lagrange-inversion argument as \Cref{prop:colored-matching-duality}.
\end{proof}

For $\beta=1$ and $\Lambda(f)=\mathbb E f(X)$ with $X\sim N(1,\lambda)$, Proposition~\ref{prop:symmetric-moments} recovers the colored-matching series. The affine-center formulation will also encompass the weighted return-walk family in the next subsection.

We finally note a terminological distinction. Matching-polynomial ``duality'' in the sense of Godsil and Lass concerns relations between a graph and its complement, including Fourier-transform identities \cite{Godsil1981Duality,Lass2004}. The coefficient duality \eqref{eq:M-duality} is of a different kind: it acts on the binomial series of the matching enumerators through a M\"obius substitution and an arbitrary test series $h$.

\subsection{Generalized central trinomial coefficients and weighted return walks}\label{sec:return-walks}

A second enumerative family arises from generalized central trinomial coefficients. For parameters $b,c\in\CC$, define
\begin{equation}
T_n(b,c)
:=
\coeff{x^0}\bigl(b+x+cx^{-1}\bigr)^n
=
\sum_{j=0}^{\lfloor n/2\rfloor}
\binom{n}{2j}\binom{2j}{j}b^{n-2j}c^j.
\label{eq:Tbc-definition}
\end{equation}
Equivalently, $T_n(b,c)$ is the coefficient of $x^n$ in $(x^2+bx+c)^n$. These numbers and their arithmetic and analytic properties have been studied extensively; see, for example, \cite{Sun2014,Sun2022,LiangWangWang2023}. Recent work of Kohen \cite{Kohen2025} studies sequences represented as coefficients of $x^0$ in powers of symmetric Laurent trinomials (more generally, with an additional Laurent-polynomial factor) and reduces their modular behavior to generalized central trinomial coefficients. For nonnegative integral $b$ and $c$, the coefficient-of-$x^0$ form above has a direct walk interpretation: it counts length-$n$ walks on $\mathbb Z$ returning to the origin, with steps $+1,0,-1$, with $b$ colors available for every level step and $c$ colors available for every down-step. Thus $T_n(1,1)$ is the central trinomial coefficient, counting grand Motzkin returns, while
\begin{equation}
T_n(2,1)=\binom{2n}{n},
\qquad
T_n(3,2)=D_n,
\label{eq:T-specializations}
\end{equation}
where $D_n$ is the central Delannoy number \cite{Sun2014,Sun2022}.

Define the corresponding binomial series
\begin{equation}
\Cser_{\alpha;b,c}(z)
:=
\sum_{n\ge0}\binom{\alpha}{n}T_n(b,c)z^n.
\label{eq:Cseries}
\end{equation}
Unlike the colored-matching family, the parameter $b$ now changes the M\"obius involution itself.

\begin{proposition}[Weighted return-walk duality]\label{prop:return-walk-duality}
For all $\alpha,b,c\in\CC$,
\begin{equation}
\Cser_{\alpha;b,c}\!\left(-\frac{z}{1+2bz}\right)
=
(1+2bz)^{-\alpha}\Cser_{\alpha;b,c}(z).
\label{eq:C-covariance}
\end{equation}
Consequently, for every formal power series $h(z)$ and every $r\ge0$,
\begin{equation}
\coeff{z^r}h(z)\Cser_{\alpha;b,c}(z)
=
\coeff{z^r}
h\!\left(\frac{z}{1-2bz}\right)
(1-2bz)^{r-1-\alpha}
\Cser_{\alpha;b,c}(-z).
\label{eq:C-duality}
\end{equation}
\end{proposition}

\begin{proof}
Put
\[
P(x):=b+x+cx^{-1}.
\]
Then
\begin{equation}
P(-x)=b-x-cx^{-1}=2b-P(x).
\label{eq:P-reflection}
\end{equation}
Since constant term is invariant under $x\mapsto-x$, one has, coefficientwise in $z$,
\[
\begin{aligned}
\Cser_{\alpha;b,c}\!\left(-\frac{z}{1+2bz}\right)
&=
\coeff{x^0}\left(1-\frac{zP(x)}{1+2bz}\right)^\alpha\\
&=
(1+2bz)^{-\alpha}
\coeff{x^0}\bigl(1+z(2b-P(x))\bigr)^\alpha\\
&=
(1+2bz)^{-\alpha}
\coeff{x^0}(1+zP(-x))^\alpha\\
&=
(1+2bz)^{-\alpha}\Cser_{\alpha;b,c}(z),
\end{aligned}
\]
which proves \eqref{eq:C-covariance}. Formula \eqref{eq:C-duality} is the corresponding coefficient-duality identity obtained by Lagrange inversion, equivalently the specialization $\beta=b$ of Proposition~\ref{prop:symmetric-moments} applied to the constant-term functional generated by $P(x)$.
\end{proof}

The involution
\begin{equation}
\tau_b(z):=-\frac{z}{1+2bz}
\label{eq:tau-b}
\end{equation}
satisfies $\tau_b^2=1$. It is diagonalized by
\begin{equation}
u:=\frac{z}{1+bz},
\qquad
\frac{\tau_b(z)}{1+b\tau_b(z)}=-\frac{z}{1+bz}.
\label{eq:C-parity-coordinate}
\end{equation}
Thus the horizontal-step weight $b$ has a direct structural meaning: it determines the affine center of the moment symmetry and hence the M\"obius involution.

The binomial series also has a hypergeometric form. From \eqref{eq:Tbc-definition},
\begin{align}
\Cser_{\alpha;b,c}(z)
&=
(1+bz)^\alpha
\sum_{j\ge0}
\frac{\alpha^{\underline{2j}}}{(j!)^2}
\left(\frac{cz^2}{(1+bz)^2}\right)^j
\label{eq:C-explicit}\\
&=
(1+bz)^\alpha\,
{}_2F_1\!\left(
-\frac{\alpha}{2},\frac{1-\alpha}{2};1;
\frac{4cz^2}{(1+bz)^2}
\right).
\nonumber
\end{align}
Under $\tau_b$, the coordinate $z/(1+bz)$ changes sign, so the invariance of the ${}_2F_1$ argument gives another immediate verification of \eqref{eq:C-covariance}. This contrasts with the colored-matching expansion \eqref{eq:M-series-explicit}, which involves ${}_2F_0$.

As before, put $\rho:=r-1-\alpha$ and define
\begin{equation}
(\mathscr S_{\rho,b}h)(z)
:=
(1+2bz)^\rho
h\!\left(-\frac{z}{1+2bz}\right).
\label{eq:Sb-operator}
\end{equation}
Then $\mathscr S_{\rho,b}^2=1$, and \eqref{eq:C-duality} becomes
\begin{equation}
\coeff{z^r}h(z)\Cser_{\alpha;b,c}(z)
=
(-1)^r
\coeff{z^r}(\mathscr S_{\rho,b}h)(z)\Cser_{\alpha;b,c}(z).
\label{eq:C-eigenspace-duality}
\end{equation}
In particular, the fixed vector $(1+2bz)^{\rho/2}$ yields the following recurrence.

\begin{corollary}[Generalized central trinomial recurrence]\label{cor:trinomial-recurrence}
Let $r$ be odd and $\alpha,b,c\in\CC$. Then
\begin{equation}
\sum_{k=0}^r
(2b)^{r-k}
\binom{(r-1-\alpha)/2}{r-k}
\binom{\alpha}{k}
T_k(b,c)
=0.
\label{eq:trinomial-general}
\end{equation}
In particular, for $m,n\ge0$,
\begin{equation}
\sum_{k=[n-m]_+}^{2n+1}
(-1)^k(2b)^{2n+1-k}
\binom{m+n+1}{m-n+k}
\binom{2m+k+1}{k}
T_k(b,c)
=0.
\label{eq:trinomial-mn}
\end{equation}
\end{corollary}

\begin{proof}
Take $h(z)=(1+2bz)^p$ in \eqref{eq:C-eigenspace-duality}, where $p:=(r-1-\alpha)/2$. For odd $r$, the resulting coefficient vanishes. Expanding gives \eqref{eq:trinomial-general}; the specialization $r=2n+1$, $\alpha=-2m-2$ gives \eqref{eq:trinomial-mn}.
\end{proof}

Three classical specializations are worth recording. If $T_k:=T_k(1,1)$ denotes the central trinomial coefficient, then
\begin{equation}
\sum_{k=[n-m]_+}^{2n+1}
(-1)^k2^{2n+1-k}
\binom{m+n+1}{m-n+k}
\binom{2m+k+1}{k}T_k=0.
\label{eq:central-trinomial-recurrence}
\end{equation}
Using $T_k(2,1)=\binom{2k}{k}$ gives
\begin{equation}
\sum_{k=[n-m]_+}^{2n+1}
(-1)^k4^{2n+1-k}
\binom{m+n+1}{m-n+k}
\binom{2m+k+1}{k}\binom{2k}{k}=0,
\label{eq:central-binomial-recurrence}
\end{equation}
and $T_k(3,2)=D_k$ gives
\begin{equation}
\sum_{k=[n-m]_+}^{2n+1}
(-1)^k6^{2n+1-k}
\binom{m+n+1}{m-n+k}
\binom{2m+k+1}{k}D_k=0.
\label{eq:delannoy-recurrence}
\end{equation}
The literature on $T_n(b,c)$ contains many three-term recurrences, congruences, and transform identities \cite{Sun2014,LiangWangWang2023}; we are not aware of the binomial-kernel family \eqref{eq:trinomial-general}--\eqref{eq:trinomial-mn} or of the arbitrary-test-series transform \eqref{eq:C-duality} in that literature.

The polynomial dependence on $c$ again yields a layerwise refinement. The coefficient of $c^j$ in $T_k(b,c)$ is
\begin{equation}
\binom{k}{2j}\binom{2j}{j}b^{k-2j},
\label{eq:return-layer}
\end{equation}
which counts weighted return walks with exactly $j$ up-steps and $j$ down-steps. Extracting $c^j$ from \eqref{eq:trinomial-general} gives, after removing the factors independent of $k$,
\begin{equation}
\sum_{k=2j}^r
2^{r-k}
\binom{(r-1-\alpha)/2}{r-k}
\binom{\alpha}{k}
\binom{k}{2j}
=0,
\qquad r\ \text{odd}.
\label{eq:universal-layer-recurrence}
\end{equation}
Up to a factor depending only on $j$, exactly the same identity underlies the fixed-edge matching refinement \eqref{eq:fixed-edge-recurrence}. Indeed, matchings with $j$ edges contribute $\binom{k}{2j}(2j-1)!!$, while return walks with $j$ up-down pairs contribute $\binom{k}{2j}\binom{2j}{j}$. The two enumerative families therefore share a universal layerwise cancellation even though their global generating functions and combinatorial constructions are different.

There is also a moment interpretation. For $c>0$, the constant-term functional may be represented by averaging
\begin{equation}
X:=b+2\sqrt c\cos\theta,
\qquad 0\le\theta<2\pi,
\label{eq:arcsine-moment}
\end{equation}
so that $T_n(b,c)=\mathbb E(X^n)$ for the shifted arcsine law. The symmetry $X\overset d=2b-X$ is precisely \eqref{eq:P-reflection}, and Proposition~\ref{prop:symmetric-moments} recovers \eqref{eq:C-covariance} with center $\beta=b$. Thus the two main enumerative examples correspond to distinct symmetric moment laws: a Gaussian law for colored matchings and an arcsine law for weighted return walks.

\begin{remark}[Motzkin numbers]
The same affine-center principle also applies to ordinary Motzkin numbers. Since
\[
M_n=\sum_{j=0}^{\lfloor n/2\rfloor}\binom{n}{2j}C_j,
\]
where $C_j$ is the $j$th Catalan number, the Motzkin sequence is the moment sequence of a distribution symmetric about $1$ whose centered even moments are Catalan numbers. It therefore satisfies the center-$1$ covariance of Proposition~\ref{prop:symmetric-moments}. We do not develop this third example here, since the generalized central trinomial family is structurally richer: its parameter $b$ varies the M\"obius involution itself. See \cite{Sun2022} for the close relation between Motzkin and central trinomial numbers.
\end{remark}

\section{Hilbert series and algebraic-combinatorial realizations}\label{sec:hilbert}

For integral weights, the covariance spaces in Theorem~\ref{thm:structural-duality} coincide with the graded spaces $\mathcal F_s$ studied by Herbig, Herden, and Seaton, after the identification $s=-\alpha$ \cite{HerbigHerdenSeaton2015,HerbigHerdenSeaton2021}. Their work connects the same M\"obius functional equation with Hilbert-series reciprocity. We record the resulting coefficient-duality consequence.

\subsection{Gorenstein Hilbert series}

The general covariance principle also applies naturally to Hilbert series. Let $A$ be a positively graded Cohen--Macaulay domain of Krull dimension $d$, with Hilbert series $\operatorname{Hilb}_A(t)$ and $a$-invariant $a$. Stanley's Gorenstein functional equation \cite{HerbigHerdenSeaton2021} can be written
\begin{equation}
\operatorname{Hilb}_A(t^{-1})=(-1)^d t^{-a}\operatorname{Hilb}_A(t).
\label{eq:Stanley-Hilbert}
\end{equation}
Herbig, Herden, and Seaton show that, after setting
\begin{equation}
\varphi_A(z):=z^d\operatorname{Hilb}_A(1-z),
\qquad s:=-(a+d),
\label{eq:Hilbert-phi}
\end{equation}
this is equivalent to $\varphi_A\in\mathcal F_s$, i.e. to the covariance law \eqref{eq:Fcov} with $\alpha=-s$ \cite{HerbigHerdenSeaton2021}. Since both the $a$-invariant $a$ and the dimension $d$ are integers, this realization supplies the integral weight $\alpha=a+d$; this is a genuine restriction of the Hilbert-series application, not of the general formal theory.

\begin{corollary}[Coefficient duality for Gorenstein Hilbert series]\label{cor:Hilbert-duality}
If $A$ is as above and is Gorenstein, then for every formal power series $h$ and every $n\ge0$,
\begin{equation}
\coeff{z^n}h(z)\varphi_A(z)
=
\coeff{z^n}h\!\left(\frac{z}{1+z}\right)
(1+z)^{n-1+s}\varphi_A(-z).
\label{eq:Hilbert-duality}
\end{equation}
\end{corollary}

\begin{proof}
Apply Theorem~\ref{thm:structural-duality} with $F=\varphi_A$ and $\alpha=-s$.
\end{proof}

For $h=1$, equation \eqref{eq:Hilbert-duality} is another coefficient form of the linear constraints on the Laurent coefficients studied in \cite{HerbigHerdenSeaton2021}. Allowing arbitrary $h$ packages those constraints into the arbitrary-test-function form of the general duality. This connection suggests that choices of $h$ adapted to a particular Hilbert series may yield compact Gorenstein obstructions or identities beyond the Bernoulli setting.

\begin{corollary}[Simplicial spheres and Dehn--Sommerville]\label{cor:Dehn-Sommerville}
Let $\Delta$ be a $(d-1)$-dimensional homology sphere with $h$-polynomial
\[
h_\Delta(t)=\sum_{j=0}^d h_jt^j,
\]
and put $\Phi_\Delta(z):=h_\Delta(1-z)$. Then $\Phi_\Delta$ is M\"obius-covariant of weight $d$, and for every nonnegative integer $r$ and every formal power series $g$,
\begin{equation}
\coeff{z^r}g(z)h_\Delta(1-z)
=
\coeff{z^r}
g\!\left(\frac{z}{1+z}\right)
(1+z)^{r-1-d}h_\Delta(1+z).
\label{eq:sphere-duality}
\end{equation}
\end{corollary}

\begin{proof}
The face ring of $\Delta$ is Gorenstein*, and the Dehn--Sommerville symmetry is
\[
h_\Delta(t)=t^d h_\Delta(t^{-1});
\]
see, for example, \cite{Masuda2005}. Substituting $t=1-z$ gives
\[
\Phi_\Delta\!\left(-\frac{z}{1-z}\right)
=(1-z)^{-d}\Phi_\Delta(z),
\]
so \eqref{eq:sphere-duality} follows from Theorem~\ref{thm:structural-duality}.
\end{proof}

For $g=1$, \eqref{eq:sphere-duality} is equivalent to the usual palindromicity $h_j=h_{d-j}$; the arbitrary-$g$ form packages its coefficient consequences in the same way as \eqref{eq:main-duality}. Thus the Gorenstein specialization includes a familiar object of algebraic combinatorics without requiring any Bernoulli input.

\section{Further directions}
\label{sec:further-directions}

\subsection{Relation with formal groups}

The coordinate
\[
u(z):=\frac{z}{2-z}
\]
does more than diagonalize the involution. If one transports the additive formal group through $u$, defining
\begin{equation}
x\oplus y
:=
u^{-1}(u(x)+u(y)),
\label{eq:formalgroup}
\end{equation}
then
\begin{equation}
\sigma(z)=u^{-1}(-u(z))
\label{eq:formal-inverse}
\end{equation}
is precisely the formal inverse for this group law.

This observation suggests that the M\"obius duality may admit a formulation in the language of one-dimensional commutative formal groups. Satoh's extension of Zagier's proof of Kaneko's recurrence to $q$-Bernoulli numbers attached to formal groups provides evidence that such a direction may be fruitful \cite{Satoh2000}. More broadly, universal Bernoulli polynomials associated with the Lazard formal group and their congruences have been developed by Tempesta and others; see \cite{Tempesta2015}. The problem below should therefore be understood as asking for a coefficient-duality structure inside an existing formal-group Bernoulli theory, rather than for a first formal-group generalization of Bernoulli numbers.

\begin{problem}
Let $F$ be a one-dimensional commutative formal group with formal logarithm $\ell_F$ and inverse
\[
\iota_F(z):=\ell_F^{-1}(-\ell_F(z)).
\]
Construct natural Bernoulli-type series $\Bser_{\alpha,F}(z)$ and a multiplicative cocycle $J_F(z)^\alpha$ satisfying a covariance law of the form
\begin{equation}
\Bser_{\alpha,F}(\iota_F(z))
=
J_F(z)^{-\alpha}
\Bser_{\alpha,F}(z),
\label{eq:formal-group-problem}
\end{equation}
and determine the corresponding coefficient duality.
\end{problem}

The M\"obius covariance studied here corresponds to a particularly simple rational formal-group coordinate.

\subsection{q-analogues}

Satoh's $q$-Bernoulli numbers attached to a formal group generalize Kaneko's recurrence by abstracting Zagier's involution \cite{Satoh2000}. It would be interesting to determine whether the general coefficient duality of Theorem~\ref{thm:structural-duality}, rather than merely its Kaneko specialization, has a $q$-analogue.

The involution formulation suggests that the correct first question is not to search directly for a $q$-analogue of the coefficient formula, but for a $q$-deformation of the underlying covariance law. Once such a covariance law is found, a suitable $q$-Lagrange inversion formula may provide the corresponding duality.

\subsection{Multisection and lacunary recurrences}

The general weighted duality \eqref{eq:weighted-duality} gives a one-parameter family of identities. The special value $c=2$ is distinguished because the transformed quadratic becomes $1-z^2$, producing the parity collapse used in \Cref{sec:bernoulli}. Root-of-unity averaging gives analogous filters for higher residue classes. We record the cubic case as a proof of concept.

\begin{proposition}[Cubic multisection]\label{prop:cubic-multisection}
Let $F$ be M\"obius-covariant of weight $\alpha$, define $f_k:=\coeff{z^k}F(z)$, let $r$ be a nonnegative integer, and put
\[
p:=\frac{r-1-\alpha}{2}.
\]
For $s\in\{0,1,2\}$ one has
\begin{equation}
\sum_{\substack{0\le k\le r\\ k\equiv s\ ({\rm mod}\ 3)}}
(-1)^k\binom{p}{r-k}f_k
=
\coeff{z^r}F(z)
\sum_{\substack{\ell\ge0\\ \ell\equiv r-s\ ({\rm mod}\ 3)}}
\binom{p}{\ell}z^\ell(1-z)^{2p-\ell}.
\label{eq:cubic-multisection}
\end{equation}
Only terms with $\ell\le r$ contribute to the coefficient on the right.
\end{proposition}

\begin{proof}
Let $\omega:=e^{2\pi i/3}$. Apply \Cref{prop:weighted-duality} with $c=1,\omega,\omega^2$, multiply the identity corresponding to $c=\omega^j$ by $\omega^{js}$, and average over $j=0,1,2$. On the left, the factor $\omega^{-jk}$ coming from $F(-z/\omega^j)$ leaves precisely the terms with $k\equiv s\pmod 3$. On the right, use
\[
1+(\omega^j-2)z-(\omega^j-1)z^2
=(1-z)\bigl((1-z)+\omega^jz\bigr)
\]
and expand formally as
\[
\bigl((1-z)+\omega^jz\bigr)^p
=\sum_{\ell\ge0}\binom{p}{\ell}
\omega^{j\ell}z^\ell(1-z)^{p-\ell}.
\]
The root-of-unity filter
\[
\frac13\sum_{j=0}^2\omega^{j(s-r+\ell)}
=
\begin{cases}
1,&\ell\equiv r-s\pmod3,\\
0,&\text{otherwise},
\end{cases}
\]
gives \eqref{eq:cubic-multisection}.
\end{proof}

For a small explicit instance, suppose $p=6$ (equivalently, $\alpha=r-13$) and choose the residue class $s\equiv r\pmod3$. Then \eqref{eq:cubic-multisection} becomes the genuinely lacunary identity
\begin{equation}
(-1)^r\bigl(f_r-20f_{r-3}+f_{r-6}\bigr)
=
\coeff{z^r}F(z)
\left((1-z)^{12}+20z^3(1-z)^9+z^6(1-z)^6\right),
\label{eq:cubic-example}
\end{equation}
with the convention $f_j=0$ for $j<0$. Thus third-root averaging isolates one residue class of coefficients exactly as $c=2$ isolates parity. Formula \eqref{eq:cubic-multisection} can be specialized to the Bernoulli and enumerative realizations developed above; further simplification depends on the additional sparsity or structure of the chosen coefficient sequence.

There is already a substantial theory of lacunary and shortened Bernoulli recurrences, including the work of Agoh and Dilcher \cite{AgohDilcher2007,AgohDilcher2009} and Agoh's generalized-character recurrences \cite{Agoh2017}. A natural next problem is to determine which of those known recurrences are recovered from higher root-of-unity multisections of \eqref{eq:weighted-duality}, and whether the method produces new residue-class patterns.

\section{Concluding remarks}\label{sec:conclusion}

The coefficient identity that motivated this work first appeared in the special setting of formal Bernoulli series. The structural characterization in Theorem~\ref{thm:structural-duality} reverses that perspective: the primary object is the M\"obius covariance law, and the Bernoulli series is one distinguished realization among many. Covariance, arbitrary-test-function coefficient duality, and parity after the coordinate $z/(2-z)$ are equivalent descriptions of the same formal symmetry.

This viewpoint separates the universal mechanism from the source-specific input. Once a covariant series is available, the eigenspace principle, anti-invariant-coordinate construction, coordinate-convolution duality, and weighted duality produce coefficient identities automatically. The Catalan specialization in Section~\ref{sec:motivating-generalizations} then gives universal convolution arrays, Chebyshev evaluations, and a Ramanujan-type summation formula whose Bernoulli member is the original half-integer power-sum identity. The remaining examples in the paper differ mainly in how the covariance is generated:
\begin{itemize}[leftmargin=2.2em]
\item Bernoulli series arise from reflection of Bernoulli polynomials and yield the recurrences of Chang--Ha, Kaneko, and Momiyama, together with the Chang--Ha Genocchi recurrence \cite[Theorem~3(b)]{ChangHa2001} and a route to Kummer's congruence.
\item Reflection-symmetric Appell sequences produce paired covariance laws, higher-order Bernoulli and Euler realizations, and an affine Euler variant recovering the Euler recurrence of Chang and Ha.
\item Colored matchings arise at an Appell fixed point and admit a Gaussian-moment interpretation; their recurrence family refines by the exact number of selected edges.
\item Generalized central trinomial coefficients arise from a Laurent-polynomial reflection and describe weighted return walks; their specializations include central trinomial, central binomial, and Delannoy numbers, again with a fixed-layer refinement.
\item Gorenstein Hilbert series arise from Stanley-type reciprocity, with simplicial spheres giving the Dehn--Sommerville specialization.
\end{itemize}

The two enumerative examples are especially instructive. Their global generating functions are different---Gaussian/$\,{}_2F_0$ for colored matchings and arcsine/$\,{}_2F_1$ for weighted return walks---yet their fixed-layer recurrences reduce to the same universal binomial cancellation. The symmetric-moment formulation identifies affine reflection as the common explanation.

The literature comparison also clarifies the role of the present results. M\"obius covariance and reciprocal Appell symmetry belong to established theories; what is emphasized here is the coefficient-extraction layer supplied by Lagrange inversion and the systematic use of the resulting duality identities as a recurrence-producing mechanism. The Catalan--M\"obius and Chebyshev transforms, the generalized Ramanujan summation formula, the weighted duality, the enumerative realizations, and the layerwise cancellations suggest further arithmetic and combinatorial applications. Formal groups and $q$-analogues remain natural directions for extending the framework, while the cubic multisection above suggests a systematic study of higher root-of-unity filters and their relation to known lacunary recurrences.

\section*{Acknowledgements}

The author acknowledges the use of OpenAI's ChatGPT as an assistive tool in the preparation of this manuscript, including literature discovery, exploratory derivations, symbolic and computational checks, organization of the exposition, and language and LaTeX editing. All mathematical statements, proofs, citations, assessments of prior work, and final wording were reviewed by the author, who assumes full responsibility for the content of the manuscript.

\end{document}